\documentclass[a4paper]{amsart}
\usepackage{amsmath,amsthm,amssymb,latexsym,epic,bbm,comment,mathrsfs}
\usepackage{graphicx,enumerate,stmaryrd,xcolor}
\usepackage[all,2cell,knot]{xy}
\xyoption{2cell}

\newtheorem{theorem}{Theorem}
\newtheorem{lemma}[theorem]{Lemma}
\newtheorem{corollary}[theorem]{Corollary}
\newtheorem{proposition}[theorem]{Proposition}
\newtheorem{conjecture}[theorem]{Conjecture}
\newtheorem{example}[theorem]{Example}

\newtheorem{remark}[theorem]{Remark}
\usepackage[all]{xy}
\usepackage[active]{srcltx}
\usepackage[parfill]{parskip}
\usepackage{enumerate}

\colorlet{MAGENTA}{magenta}
\newcommand{\tto}{\twoheadrightarrow}

\begin{document}

\title[Projective functors on holonomic simple modules]
{Applying projective functors to \\arbitrary holonomic simple modules}
\author[M.~Mackaay, V.~Mazorchuk and V.~Miemietz]{Marco Mackaay, 
Volodymyr Mazorchuk and Vanessa Miemietz}

\begin{abstract}
We prove that applying a projective functor to a holonomic simple
module over a semi-simple finite dimensional complex Lie algebra
produces a module that has an essential semi-simple submodule of
finite length. This implies that holonomic 
simple supermodules over certain
Lie superalgebras are quotients of modules that are induced
from simple modules over the even part. We also provide some
further insight into the structure of Lie algebra modules that are
obtained by applying projective functors to simple modules.
\end{abstract}

\maketitle

\section{Motivation and description of the results}\label{s1}

\subsection{Motivation from Lie superalgebras}\label{s1.1}

Let $\mathfrak{g}$ be a semi-simple (or reductive) finite dimensional Lie algebra
over $\mathbb{C}$. Let $\mathfrak{s}=\mathfrak{s}_0\oplus \mathfrak{s}_1$ 
be a finite dimensional complex Lie superalgebra such that 
$\mathfrak{s}_0\cong \mathfrak{g}$. One of the basic representation-theoretic
problems for $\mathfrak{s}$ is the classification of simple 
$\mathfrak{s}$-supermodules, see, for example, \cite{Ma10,CM21,CCM21}. 
A natural way to address this problem is to look for some connection 
between simple $\mathfrak{s}$-supermodules and simple $\mathfrak{g}$-modules.
In \cite{Ma10,CM21,CCM21} this approach was successfully applied to
reduce the problem of classification of simple $\mathfrak{s}$-supermodules 
to that of simple $\mathfrak{g}$-modules. The latter problem is 
very difficult, the only known case in which it is ``solved''
(in the sense that it is reduced to the problem of classification of 
equivalence classes of irreducible elements over a certain 
non-commutative principal ideal domain)
is $\mathfrak{sl}_2$, see \cite{Bl}. 

The module categories $\mathfrak{s}$-Mod and $\mathfrak{g}$-Mod are connected
by the usual induction and restriction functors. Moreover, the latter two
functors are not only adjoint in the obvious way (that is, induction is left
adjoint to restriction), but they are also biadjoint, up to parity shift
(that is, induction is right adjoint to restriction, up to parity shift
that depends on the parity of the dimension of $\mathfrak{s}_1$).
This is equivalent to saying that induction is isomorphic to
coinduction, up to parity shift. Since the
universal enveloping algebra $U(\mathfrak{g})$ is noetherian and
the universal enveloping algebra $U(\mathfrak{s})$ is finite over $U(\mathfrak{g})$,
every simple $U(\mathfrak{s})$-supermodule $S$ has a simple
quotient, say $L$, when considered as a $U(\mathfrak{g})$-module.
By adjunction, it follows that $S$ is a submodule of a module
which is coinduced from a simple $U(\mathfrak{g})$-module.
That is a very natural fact. 

Now we recall that induction and coinduction coincide, up to a parity
shift. It follows that  $S$ is a submodule of a module
which is induced from a simple $U(\mathfrak{g})$-module. 
It would be more natural to expect $S$ to
be a quotient of a module induced from a simple $U(\mathfrak{g})$-module.
However, it seems that there is no easy argument for why that should be the case.
This property is an essential ingredient in \cite{CCM21} where 
the claim is proved in type $A$ using very specific type $A$ properties
established in \cite{MS08,MM16}.

The idea is that, in order to use the correct adjunction, we need to show that 
$S$, when restricted to  $U(\mathfrak{g})$, has a simple submodule. Note that
we already know that $S$ is a submodule of an induced simple module. Therefore
it is enough to show that any $U(\mathfrak{s})$-supermodule which is induced
from a simple $U(\mathfrak{g})$-module,
when restricted back to $U(\mathfrak{g})$, has finite type socle, that 
is, it has an essential semi-simple submodule of finite length.

At the level of $U(\mathfrak{g})$-modules, the composition of induction to 
$U(\mathfrak{s})$ followed by restriction back to $U(\mathfrak{g})$
can be described as tensoring 
with a finite dimensional  $U(\mathfrak{g})$-module, namely, with
$\bigwedge \mathfrak{s}_1$. This naturally
leads to the formulation of our main result in Theorem~\ref{thm1} below.

\subsection{Main result}\label{s1.2}
Recall that a simple module  over a
finitely generated associative algebra of finite
Gelfand-Kirillov dimension is called
{\em holonomic} provided that it has the minimal possible 
Gelfand-Kirillov dimension among all simple modules
with the same annihilator.

The main result of this paper is the following statement:
\vspace{2mm}

{\bf Theorem~\ref{thm1}.} 
{\em Let $\mathfrak{g}$ be a semi-simple finite dimensional Lie algebra
over $\mathbb{C}$. Let $L$ be a holonomic simple $\mathfrak{g}$-module and
let $V$ be a finite dimensional $\mathfrak{g}$-module. Then the 
$\mathfrak{g}$-module $V\otimes_{\mathbb{C}} L$ has an
essential semi-simple submodule of finite length.}
\vspace{2mm}

As an immediate corollary, it follows that any holonomic
simple $\mathfrak{s}$-supermodule
is, indeed, a quotient of a module which is induced from a simple $\mathfrak{g}$-module.

In type $A$, the assertion of Theorem~\ref{thm1} is true for all simple
$\mathfrak{g}$-modules, not necessarily holonomic ones, see \cite[Theorem~23]{CCM21}. Of course,
we expect the assertion of Theorem~\ref{thm1} to be true for all simple
$\mathfrak{g}$-modules in all types. However, at the moment we do not see how to prove that.

\subsection{Structure of $V\otimes_{\mathbb{C}} L$}\label{s1.3}

The main difficulty in proving the main result lies in the fact that the module
$V\otimes_{\mathbb{C}} L$, in general, while being noetherian, 
does not have to be artinian, see \cite{Sta} for an example. At the same time,
in all known ``natural'' examples, for instance, if we assume that 
$L$ belongs to the BGG category $\mathscr{O}$, see \cite{BGG,Hu},
or to the category of weight modules with finite dimensional weight spaces, see \cite{Mat},
or to the category of Gelfand-Zeitlin modules, see \cite{FO,EMV},
the module $V\otimes_{\mathbb{C}} L$ has finite length. Therefore in our
approach we cannot really rely on the intuition developed during
the study of these classical categories of modules. We need
to understand the structure of $V\otimes_{\mathbb{C}} L$ in very
abstract terms and in a situation where we lack easy computable examples.

Possible subquotients of interest in $V\otimes_{\mathbb{C}} L$ 
split naturally into three categories:
\begin{itemize}
\item simple subquotients whose Gelfand-Kirillov dimension equals
$\mathrm{GKdim}(L)$;
\item simple subquotients whose Gelfand-Kirillov dimension is strictly 
smaller than $\mathrm{GKdim}(L)$;
\item non-simple subquotients that we call {\em strange},
see Subsection~\ref{s4.4} for details, and which are
defined by the property that they have Gelfand-Kirillov dimension 
$\mathrm{GKdim}(L)$ but they do not have any simple 
subquotient of Gelfand-Kirillov dimension $\mathrm{GKdim}(L)$.
\end{itemize}
Our proof of the main result essentially reduces to the statement that
$V\otimes_{\mathbb{C}} L$ cannot have strange submodules. 

A major part of the paper is devoted to taking a closer  look at the 
general structure of $V\otimes_{\mathbb{C}} L$. As mentioned above, 
this module might fail to have finite length.  However, one can define
a natural Serre subquotient of the category of all $\mathfrak{g}$-module
in which $V\otimes_{\mathbb{C}} L$ does have finite length, see
Subsection~\ref{s9.3}. The structure of $V\otimes_{\mathbb{C}} L$ as
an object of this Serre subquotient is similar in spirit to what was called 
the {\em rough structure} of generalized Verma  modules in \cite{MS08}.

The correct setup for the study of modules of the form 
$V\otimes_{\mathbb{C}} L$ is to combine them all into a 
certain birepresentation of the bicategory of projective functors
associated to the algebra $\mathfrak{g}$. In the case when $L$
is a simple highest weight module, these birepresentations appear
frequently and were studied extensively in many papers, see
\cite{MS08,Ma10b,MMMTZ2} and references therein. In the general
case, it is natural to expect that the corresponding 
birepresentations behave similarly to the case of highest weight
modules. One possible direction of this expectation is formulated, 
in precise terms, in Conjecture~\ref{conj7} in Subsection~\ref{s3.1}. 
This conjecture asserts that the birepresentation in question is
{\em simple transitive} in the terminology of \cite{MM16}.

In type $A$, the conjecture is proved in Subsection~\ref{s4.3}.
In fact,  in type $A$, we establish an equivalence between
a birepresentation in the general case and a birepresentation
in the highest weight case. We do not expect such an equivalence
for other types, in general, as we know from \cite{MMMTZ2}  that,
outside type $A$, not all simple transitive birepresentations of projective 
functors can be modeled naively using highest weight modules
(a modeling via highest weight modules is possible,  but it 
requires an upgrade to the level of (co)algebra $1$-morphisms and 
the corresponding categories of (co)modules). It would be really
interesting if, in general type, each simple transitive
birepresentation of projective functors turned out to be constructible
directly starting from some simple (but not necessarily 
highest weight) module. At the moment, we do not know whether
this is true or not and  where to look for such simple modules.

Another interesting aspect of the structure of 
$V\otimes_{\mathbb{C}} L$ which we analyze is the following:
There is a natural pre-order $\triangleright$ on the set of
isomorphism classes of simple $\mathfrak{g}$-modules given by
$L\triangleright L'$ provided that $L'$  is a quotient of
$V\otimes_{\mathbb{C}} L$, for some $V$. In 
Conjecture~\ref{conj7-2} we predict that $\triangleright$
is, in fact, an equivalence relation. Again, in type $A$,
we can prove this conjecture, see Subsection~\ref{s4.2}.
We find it very surprising that this, intuitively very 
natural expectation, seems to be very non-trivial in reality
and even in type $A$ its proof requires quite heavy machinery.
At the moment we do not know how to prove this conjecture in general.

\subsection{Structure of the paper}\label{s1.4}

The paper is organized as follows: in Section~\ref{s2new} we collect
all necessary preliminaries. In Sections~\ref{s3}, \ref{s4} and \ref{s9}, we
study the general structure of modules of the form $V\otimes_{\mathbb{C}} L$
in more detail. Section~\ref{s3} contains all necessary preliminaries in
order to formulate Conjecture~\ref{conj7} (in Subsection~\ref{s3.1}) and
Conjecture~\ref{conj7-2} (in Subsection~\ref{s3.2}). The two conjectures
are compared in Subsection~\ref{s3.4}.
Section~\ref{s4} contains  proofs of
Conjecture~\ref{conj7} and Conjecture~\ref{conj7-2} in type $A$.
The main result is proved in  
Section~\ref{s2} and is extended beyond the holonomic case
to some further special cases
outside type $A$ in Section~\ref{s2.17s}.
Section~\ref{s9} is devoted to the study of strange subquotients
of modules of the form $V\otimes_{\mathbb{C}} L$ as well as 
Serre subquotients of the category of $\mathfrak{g}$-modules
in which modules of the form $V\otimes_{\mathbb{C}} L$ have finite length.

\vspace{1mm}

\subsection*{Acknowledgments}
The first author is partially supported by Funda{\c c}{\~a}o para a 
Ci{\^ e}ncia e a Tecnologia (Portugal),
project UID/MAT/04459/2013 (Center for Mathematical Analysis, Geometry 
and Dynamical Systems - CAMGSD).
The second author is partially supported by the Swedish Research Council.
The first and the third authors thank Uppsala University for hospitality
during their visit in August 2023 when a major part of the research 
reported in this paper was done.
The second author thanks Hankyung Ko, Kevin Coulembier and Chih-Whi Chen 
for stimulating discussions.

\section{Preliminaries}\label{s2new}

\subsection{Category $\mathscr{O}$}\label{s2.2}

Fix a triangular decomposition $\mathfrak{g}=\mathfrak{n}_-\oplus 
\mathfrak{h}\oplus \mathfrak{n}_+$ of $\mathfrak{g}$. Associated
to this decomposition we have the corresponding BGG category $\mathscr{O}$,
see \cite{BGG,Hu}. Simple objects in $\mathscr{O}$ are the simple
highest weight modules $L(\lambda)$, where $\lambda\in\mathfrak{h}^*$.
The module $L(\lambda)$ is the unique simple quotient of the
Verma module $\Delta(\lambda)$.

Let $Z(\mathfrak{g})$ be the center of the universal enveloping algebra
$U(\mathfrak{g})$ of $\mathfrak{g}$. Then 
$\mathscr{O}$ decomposes into a direct sum of $\mathscr{O}_\chi$,
where $\chi:Z(\mathfrak{g})\to \mathbb{C}$ is a central character of $U(\mathfrak{g})$.
The category $\mathscr{O}_\chi$ is a full subcategory of $\mathscr{O}$
consisting of all modules on which  the  kernel of $\chi$ acts 
locally nilpotently. An important fact about
$\mathscr{O}_\chi$ is that it is always non-zero. In other words, any character
of $Z(\mathfrak{g})$ is realizable as the central character of some
$L(\lambda)$, see \cite[Section~7]{Di}. For a fixed $\chi$, the set 
of all $\lambda$ such that $L(\lambda)$ has central character $\chi$
is an orbit of the Weyl group $W$ of $(\mathfrak{g},\mathfrak{h})$ on $\mathfrak{h}^*$
with respect to the so-called {\em dot-action}, that is the shift of the 
natural action by half the sum of all positive roots.

If $L$ is some simple $\mathfrak{g}$-module (not necessarily in 
category $\mathscr{O}$), then the annihilator
$\mathrm{Ann}_{U(\mathfrak{g})}(L)$ of $L$ in $U(\mathfrak{g})$
is a primitive ideal and it is realizable as the annihilator of
some $L(\lambda)$, see \cite{Du}. 

For $\lambda\in \mathfrak{h}^*$, we have the 
indecomposable projective cover
$P(\lambda)$ of $L(\lambda)$ in $\mathscr{O}$ and the indecomposable 
injective envelope $I(\lambda)$ of $L(\lambda)$ in $\mathscr{O}$.

We denote by $\mathbf{R}\subset \mathfrak{h}^*$ the root system of 
$\mathfrak{g}$ with respect to $\mathfrak{h}$. Our choice of 
triangular decomposition above gives rise to the decomposition of
$\mathbf{R}$ into a disjoint union of positive roots $\mathbf{R}_+$
and negative roots $\mathbf{R}_-$. We denote by $\pi$ the 
corresponding basis of $\mathbf{R}$. We also denote by
$\Xi$ the root lattice $\mathbb{Z}[\mathbf{R}]$.

By definition, a weight in $\mathfrak{h}^*$ is {\em integral} if it is a weight of some finite 
dimensional $\mathfrak{g}$-module. We denote by $\Lambda$ the lattice of 
all integral weights. Note that $\Xi\subset \Lambda$ is a subgroup of
finite index. This index is the determinant of the Cartan matrix for
$\mathfrak{g}$, in particular, $\Xi=\Lambda$ only in types $E_8$, $F_4$ and $G_2$.

For a weight $\lambda$, we denote by $W_\lambda$
the integral Weyl group of $\lambda$, that is the subgroup
of $W$ generated by all reflections $s$ for which 
$s\cdot\lambda-\lambda$ is an integral multiple of a root. 
We also denote by $W'_\lambda$ the stabilizer of 
$\lambda$ in $W_\lambda$. We call $\lambda\in\mathfrak{h}^*$
{\em regular} if $W'_\lambda=\{e\}$. If $\lambda\in\mathfrak{h}^*$
is not regular, it is called  {\em singular}. We call
$\lambda$ {\em dominant} if $w\cdot\lambda\leq \lambda$,
for all $w\in W_\lambda$.

We note that the categories $\mathscr{O}_\chi$ are usually decomposable.
For $\lambda\in\mathfrak{h}^*$, we denote by $\mathscr{O}_\lambda$
the indecomposable direct summand (block) of $\mathscr{O}$
containing $L(\lambda)$. Let $\chi=\chi_{{}_\lambda}$ be the central character
of $L(\lambda)$. Consider $W\cdot\lambda$ and define on this finite set
the equivalence relation $\equiv$ as follows: for 
$\mu,\nu\in W\cdot\lambda$, set $\mu\equiv\nu$ provided
that $W_\mu\cdot\mu=W_\nu\cdot\nu$. If $\{\lambda_1,\dots,\lambda_k\}$ 
is any cross-section  of equivalence classes, then $\mathscr{O}_\chi$ decomposes 
into the direct sum $\mathscr{O}_{\lambda_1}\oplus\dots\oplus \mathscr{O}_{\lambda_k}$.

\subsection{Projective functors}\label{s2.3}

In this subsection we recall the definition and basic properties of 
projective functors, as introduced in \cite{BG}.

Let $\mathscr{M}$ denote the category of all $\mathfrak{g}$-modules
on which the action of $Z(\mathfrak{g})$ is locally finite.
Note that $\mathscr{O}\subset \mathscr{M}$. Similarly to $\mathscr{O}$,
the category $\mathscr{M}$ decomposes into a product of
the full subcategories $\mathscr{M}_\chi$, where $\chi$ is a
central character, defined as follows: $\mathscr{M}_\chi$ consists of
all objects on which the kernel of $\chi$ acts locally nilpotently.
Clearly, $\mathscr{O}_\chi\subset \mathscr{M}_\chi$.

For any finite dimensional $\mathfrak{g}$-module $V$, tensoring with
$V$ preserves both $\mathscr{M}$ and $\mathscr{O}$. A {\em 
projective functor} is an endofunctor of $\mathscr{M}$ (or $\mathscr{O}$)
which is isomorphic to a direct summand of tensoring with some $V$.
The functor $V\otimes_{\mathbb{C}}{}_-$ is biadjoint to 
$V^*\otimes_{\mathbb{C}}{}_-$, for any finite dimensional
$\mathfrak{g}$-module $V$, and hence each projective functor has a 
biadjoint projective functor. Consequently, any projective functor
is exact. Furthermore, each projective functor
is isomorphic to a direct sum of indecomposable
projective functors. Indecomposable projective functors are classified
by their restriction to $\mathscr{O}$. 

Let $\mathscr{O}'$ and $\mathscr{O}''$ be two indecomposable blocks of 
$\mathscr{O}$.  Let  $L(\lambda_1)$, $L(\lambda_2),\dots,L(\lambda_r)$ 
and $L(\mu_1)$, $L(\mu_2),\dots,L(\mu_s)$ be complete and 
irredundant lists of simples in $\mathscr{O}'$ and $\mathscr{O}''$,
respectively.  Assume that $\lambda_1$ and $\mu_1$ are the weights in the above lists
that are dominant with respect to the corresponding integral Weyl groups
(each list contains a unique such dominant weight).

Non-zero projective functors from $\mathscr{O}'$ to $\mathscr{O}''$ 
exist if and only if $\lambda_1-\mu_1$ is an integral weight.
Indecomposable projective functors from $\mathscr{O}'$ to $\mathscr{O}''$
are in bijection with those $\mu_i$ that are $W'_{\lambda_1}$-anti-dominant with respect
to the dot-action. In fact, for each such $\mu_i$, there is
a unique indecomposable projective functor, denoted $\theta_{\lambda_1,\mu_i}$, 
from $\mathscr{O}'$ to $\mathscr{O}''$ that sends $P(\lambda_1)=\Delta(\lambda_1)$ to $P(\mu_i)$.

Let now $\chi'$ and $\chi''$ be two central characters.
The above can be used to classify indecomposable projective 
functors from $\mathscr{M}_{\chi'}$ to $\mathscr{M}_{\chi''}$.
Let $\lambda$ and $\nu$ be some weights such that $L(\lambda)$ 
and $L(\nu)$ have the central characters $\chi'$ and $\chi''$, respectively. 
Without loss of generality we may assume that  $\lambda$ is dominant
with respect to its integral Weyl group. 

Projective functors from $\mathscr{M}_{\chi'}$ to $\mathscr{M}_{\chi''}$
exist if and only if $(W\cdot \nu)\cap (\lambda+\Lambda)$ is not empty.
If this condition is satisfied, we may assume
$\nu\in \lambda+\Lambda$ without loss of generality. Let 
\begin{displaymath}
(W\cdot \nu)\cap (\lambda+\Lambda)=
\{\nu_1,\nu_2,\dots,\nu_t\}.
\end{displaymath}
Then indecomposable projective functors from
$\mathscr{M}_{\chi'}$ to $\mathscr{M}_{\chi''}$ are 
exactly the functors $\theta_{\lambda,\nu_i}$, where $\nu_i$
is $W'_\lambda$-anti-dominant. We note that the Serre subcategory of
$\mathscr{O}$ generated by $L(\nu_1)$, $L(\nu_2)$,\dots,
$L(\nu_t)$ does not have to be indecomposable, and hence our
indecomposable projective functors from $\mathscr{M}_{\chi'}$ to 
$\mathscr{M}_{\chi''}$ are not classified, in the general case, by 
projective functors from an indecomposable block of $\mathscr{O}$
to an indecomposable block of $\mathscr{O}$.

\subsection{Harish-Chandra bimodules}\label{s2.4}

An alternative way to look at projective functors is using 
Harish-Chandra bimodules. A $\mathfrak{g}$-$\mathfrak{g}$-bimodule
$B$ is called a {\em Harish-Chandra bimodule} provided that it is 
finitely generated as a bimodule and the adjoint action of 
$\mathfrak{g}$ on it is locally finite with finite multiplicities.
A typical example of a Harish-Chandra bimodule is the quotient
of $U(\mathfrak{g})$ by the ideal generated by the kernel of
some central character. The category of all Harish-Chandra bimodules
is denoted $\mathscr{H}$. 

The category $\mathscr{H}$ is, naturally, a monoidal category,
where the monoidal structure is given by tensoring over $U(\mathfrak{g})$.
As a monoidal category, $\mathscr{H}$ is naturally 
$\Lambda/\Xi$-graded. Namely, for a coset $\xi\in \Lambda/\Xi$,
the corresponding homogeneous component $\mathscr{H}^\xi$ consists
of all bimodules $B$ such that, for any finite dimensional 
simple $\mathfrak{g}$-module $V$, the fact that 
the multiplicity $[B^{\mathrm{ad}}:V]>1$ implies that the
support of $V$ belongs to $\xi$. For a Harish-Chandra bimodule
$B$, we will denote by $B^\xi$ its projection onto $\mathscr{H}^\xi$.

Let us explain how this grading works. Given a 
$U(\mathfrak{g})$-$U(\mathfrak{g})$-bimodule $B$, the left action
of $U(\mathfrak{g})$ on $B$ is given by a map
$U(\mathfrak{g})\otimes_{\mathbb{C}}B\to B$,
which is a homomorphism of both left and right $\mathfrak{g}$-modules and hence also
of adjoint $\mathfrak{g}$-modules.
As $U(\mathfrak{g})$, considered as an adjoint $\mathfrak{g}$-module, 
is a direct sum of simple finite dimensional 
$\mathfrak{g}$-modules whose support belongs to $\Xi$,
the above action map restricts to 
$U(\mathfrak{g})\otimes_{\mathbb{C}}B^\xi\to B^\xi$,  for every $\xi$.
Similarly, we have a restriction of the right action map,
which implies that $B^\xi$ is, indeed, a 
$U(\mathfrak{g})$-$U(\mathfrak{g})$-subbimodule of $B$.

The above grading is motivated by the fact that the action of
projective functors on category $\mathscr{O}$ behaves ``slightly better'' 
than on other natural categories of $\mathfrak{g}$-modules. 
As was mentioned in Subsection~\ref{s2.3}, projective functors 
are uniquely determined (up to isomorphism) by the image of dominant Verma modules
in category $\mathscr{O}$. Outside category $\mathscr{O}$ it might
happen that analogues of dominant Verma modules do not exist and
two non-isomorphic projective functors map some module which we
want to understand to isomorphic modules (see 
\cite[Theorem~B]{Ch}, \cite[Section~2.3]{Jo2} and
also the connection
of this phenomenon to Kostant's problem, as was observed by
Johan K{\aa}hrstr{\"o}m and explained in \cite{MMM}). 
In such situation,
one could try to consider the action of the ``smaller'' category
$\mathscr{H}^\Xi$ and, essentially, use similar arguments
as in the case of category $\mathscr{O}$. For example, this
was done in \cite{KhM} in the context of the study of generalized
Verma modules. We are going to use this kind of trick in 
Section~\ref{s4}.

Tensoring with finite dimensional $\mathfrak{g}$-modules both on the
left and on the right preserves Harish-Chandra bimodules. Therefore
indecomposable projective functors can be viewed as summands of
the Harish-Chandra bimodules 
$V\otimes_\mathbb{C}(U(\mathfrak{g})/(\mathrm{ker}(\chi)))$,
where $\chi$ is a central character.
In fact, the indecomposable projective functors 
with domain $\mathscr{M}_\chi$
correspond exactly to 
the indecomposable summands of
$V\otimes_\mathbb{C}(U(\mathfrak{g})/(\mathrm{ker}(\chi)))$.

For two $\mathfrak{g}$-modules $M$ and $N$, we can consider 
the $\mathfrak{g}$-$\mathfrak{g}$-bimodule $\mathrm{Hom}_{\mathbb{C}}(M,N)$
and its subbimodule $\mathcal{L}(M,N)$ which consists of all elements
of $\mathrm{Hom}_{\mathbb{C}}(M,N)$ on which the adjoint action of 
$\mathfrak{g}$ is locally finite. If 
$\mathrm{Hom}_{\mathfrak{g}}(V\otimes_{\mathbb{C}}M,N)$ is finite
dimensional, for any simple finite dimensional $\mathfrak{g}$-module
$V$, and both quotients $Z(\mathfrak{g})/\mathrm{Ann}_{Z(\mathfrak{g})}(M)$ 
and $Z(\mathfrak{g})/\mathrm{Ann}_{Z(\mathfrak{g})}(N)$ are finite dimensional,
then $\mathcal{L}(M,N)$ is a Harish-Chandra bimodule, see \cite[Satz~6.30]{Ja}.
For example, this is the case if both modules $M$ and $N$ belong to category $\mathscr{O}$.

If $M=N$, we have a natural embedding of $U(\mathfrak{g})/\mathrm{Ann}_{U(\mathfrak{g})}(M)$
into $\mathcal{L}(M,M)$. A module $M$ is called {\em Kostant positive} provided
that this embedding is an isomorphism. A module $M$ is called {\em weakly
Kostant positive} provided that the natural embedding 
of $U(\mathfrak{g})/\mathrm{Ann}_{U(\mathfrak{g})}(M)$
into $\mathcal{L}(M,M)^\Xi$ is an isomorphism.

\subsection{Soergel's combinatorial description and 
the integral part}\label{s2.5}

Denote by $\mathscr{O}_{\mathrm{int}}$ the full subcategory of 
$\mathscr{O}$ that consists of all modules with integral support.
It is a direct summand of $\mathscr{O}$. If a central character
$\chi$ is such that $\mathscr{O}_\chi\cap \mathscr{O}_{\mathrm{int}}$
is non-zero, then $\mathscr{O}_\chi\subset\mathscr{O}_{\mathrm{int}}$
and $\mathscr{O}_\chi$ is an indecomposable block of $\mathscr{O}$.
For such a $\chi$, let  $\lambda$ be the unique dominant weight such that 
$L(\lambda)\in \mathscr{O}_\chi$. Note that $W=W_\lambda$
as $\lambda$ is integral.
Soergel's combinatorial 
description of $\mathscr{O}$, see \cite{So}, determines $\mathscr{O}_\chi$
uniquely, up to equivalence, in terms of the algebra 
of $W'_\lambda$-invariants in the coinvariant algebra for $W$.
Projective endofunctors between different blocks of 
$\mathscr{O}_{\mathrm{int}}$ can then be described in terms of 
induction and restriction functors between the 
coinvariant algebra for $W$ and its corresponding 
invariant subalgebras, see \cite{So3}.

For a (not necessarily integral) weight $\lambda$, the category 
$\mathscr{O}_\lambda$ is equivalent to 
an integral block of $\mathscr{O}$ for a semi-simple complex
Lie algebra corresponding to the Weyl group $W_\lambda$.
Due to Soergel's combinatorial description of projective 
functors mentioned above (see \cite{So3}), this equivalence
is compatible with the action of those projective functors
which are homogeneous of degree $\Xi$.

\subsection{Gelfand-Kirillov dimension}\label{s2.6}

In this subsection we recall basic facts about the Gelfand-Kirillov dimension 
(denoted $\mathrm{GKdim}$).
We refer to \cite{KrLe} for all details.

To each finitely generated $\mathfrak{g}$-module $M$ we can associate 
its  Gelfand-Kirillov dimension $\mathrm{GKdim}(M)\in\mathbb{Z}_{\geq 0}$ 
(which is the  degree of the polynomial that describes the growth of $M$) and
its Bernstein number $\mathrm{BN}(M)\in\mathbb{Z}_{>0}$ (which  is the 
coefficient of the leading term of that polynomial).  Contrary to the 
Gelfand-Kirillov dimension, the Bernstein number might depend on the 
choice of generators in $U(\mathfrak{g})$, so we fix such a choice for 
the remainder of the paper.

The algebra $U(\mathfrak{g})$ is a noetherian algebra of finite 
Gelfand-Kirillov dimension, namely 
$\mathrm{GKdim}(U(\mathfrak{g}))=\dim(\mathfrak{g})$. 
Therefore every simple $\mathfrak{g}$ module 
has Gelfand-Kirillov dimension at most $\dim(\mathfrak{g})$
(in reality, at most $\dim(\mathfrak{g})-\mathrm{rank}(\mathfrak{g})-1$
as $Z(\mathfrak{g})$ is a polynomial algebra in $\mathrm{rank}(\mathfrak{g})$
variables). If $\mathcal{I}$ is a primitive ideal of $U(\mathfrak{g})$
and $L$ a simple module with annihilator $\mathcal{I}$, which 
minimizes the Gelfand-Kirillov dimension in the class of all simple
$\mathfrak{g}$-modules with annihilator $\mathcal{I}$, then $L$
is called {\em holonomic}. For example, all simple modules in
category $\mathscr{O}$ are holonomic, see \cite[Subsection~10.9]{Ja}.
Non-holonomic modules do certainly exist, see \cite{Sta,Co}.
As a matter of fact, almost all (in some sense) simple modules are non-holonomic.

For a finite dimensional $\mathfrak{g}$-module $V$, we have
\begin{displaymath}
\mathrm{GKdim}(V)=0,\qquad  \mathrm{GKdim}(V\otimes_{\mathbb{C}}M)=\mathrm{GKdim}(M)
\end{displaymath}
and $\mathrm{BN}(V\otimes_{\mathbb{C}}M)=\dim(V)\cdot \mathrm{BN}(M)$.
If $0\to X\to Y\to Z\to 0 $ is a short exact sequence, then 
$\mathrm{GKdim}(Y)=\max\{\mathrm{GKdim}(X),\mathrm{GKdim}(Z)\}$.
Moreover, $\mathrm{BN}(Y)=\mathrm{BN}(X)+\mathrm{BN}(Z)$
provided that $\mathrm{GKdim}(X)=\mathrm{GKdim}(Z)$. If the 
latter condition is not satisfied, then $\mathrm{BN}(Y)$
coincides with $\mathrm{BN}(X)$ if $\mathrm{GKdim}(X)>\mathrm{GKdim}(Z)$
and $\mathrm{BN}(Y)$
coincides with $\mathrm{BN}(Z)$ if $\mathrm{GKdim}(X)<\mathrm{GKdim}(Z)$.

\begin{lemma}\label{lem-n4}
Let $L$ and $L'$ be two simple $\mathfrak{g}$-modules such that 
$\mathrm{Hom}_{\mathfrak{g}}(V\otimes_{\mathbb{C}}L,L')\neq0$
or $\mathrm{Hom}_{\mathfrak{g}}(L',V\otimes_{\mathbb{C}}L)\neq0$,
for some finite dimensional $\mathfrak{g}$-module $V$.
Then $\mathrm{GKdim}(L)=\mathrm{GKdim}(L')$.
\end{lemma}

\begin{proof}
We prove the first claim. From
$\mathrm{Hom}_{\mathfrak{g}}(V\otimes_{\mathbb{C}}L,L')\neq0$, we have
$\mathrm{GKdim}(L')\leq \mathrm{GKdim}(L)$.
At the same time, by adjunction, we have
$\mathrm{Hom}_{\mathfrak{g}}(L,V^*\otimes_{\mathbb{C}}L')\neq0$.
Hence $\mathrm{GKdim}(L)\leq \mathrm{GKdim}(L')$ and thus
$\mathrm{GKdim}(L)=\mathrm{GKdim}(L')$. 
\end{proof}

Due to the 
additivity of the Bernstein number, 
$V\otimes_{\mathbb{C}}L$ has a maximal semi-simple
quotient and this quotient has finite length (and is always non-zero).
This maximal semi-simple quotient is usually called the {\em top}
of $V\otimes_{\mathbb{C}}L$. It further follows
that the module $V\otimes_{\mathbb{C}}L$ has a maximal semi-simple
submodule and this submodule has finite length (but, potentially,
may be zero). Theorem~\ref{thm1}, in fact, shows that this 
maximal semi-simple submodule is essential and hence 
is the {\em socle} of $V\otimes_{\mathbb{C}}L$. 
Note that, in general, $V\otimes_{\mathbb{C}}L$ is not of finite length,
see \cite{Sta}. However, we will show in Subsection~\ref{s9.3} that 
$V\otimes_{\mathbb{C}}L$ has
finite {\em rough length} in the sense of \cite{MS08}.

\subsection{Kazhdan-Lusztig combinatorics}\label{s2.65}

To each pair $(W',S')$, where $W'$ is a Weyl group and
$S'$ a fixed set of simple reflections in $W'$, we have the 
associated Hecke  algebra $\mathbb{H}=\mathbb{H}(W',S')$,
which  is an algebra over $\mathbb{Z}[v,v^{-1}]$,
defined by substituting the  relation $(s-e)(s+e)=0$, for $s\in S'$,
in the Coxeter presentation of $W'$, by the relation
$(H_s+v)(H_s-v^{-1})=0$ and keeping the braid relations, 
see e.g. \cite{So4}. 
It has the standard basis $\{H_w\,:\, w\in W\}$
and the Kazhdan-Lusztig basis $\{\underline{H}_w\,:\, w\in W\}$.

For $x,y\in W'$ we set $x\geq_L y$ provided that 
there is $z\in W'$ such that $\underline{H}_x$ appears with
a non-zero coefficient in $\underline{H}_z\underline{H}_y$. 
This defines a pre-order on $W'$ called the {\em KL-left pre-order}.
Equivalence classes with respect to it are called {\em KL-left cells}.
The {\em KL-right pre-order} $\geq_R$ and the corresponding 
{\em KL-right cells} are defined similarly using multiplication 
on the right. The {\em KL-two-sided pre-order} $\geq_J$ and the 
corresponding  {\em KL-two-sided cells} are defined similarly using 
multiplication  on both sides. 

If $\lambda\in\mathfrak{h}^*$ is regular 
and dominant, then the results of \cite{BV1,BV2}
imply that sending $w\in W_\lambda$ to the annihilator 
of $L(w\cdot\lambda)$ gives rise to a bijection between 
the KL-left cells in $W_\lambda$ and the primitive ideals in
$U(\mathfrak{g})$ containing $\mathrm{Ker}(\chi_{{}_\lambda})$.

The function that assigns to $w\in W_\lambda$ the Gelfand-Kirillov
dimension of $L(w\cdot\lambda)$ is constant on KL-two-sided cell.
Indeed, that this function is constant on KL-left cells follows from
\cite[Satz~10.9]{Ja} combined with the fact mentioned above that
annihilators of simple modules are constant on KL-left cells.
That the function is constant on KL-right cells follows from
Lemma~\ref{lem-n4} combined with the fact that simple modules
inside the same KL-right cell can be obtained from each other 
by applying projective functors and taking subquotients.

If $W''$ is a parabolic subgroup of $W_\lambda$, 
$w''_0$ the longest element in $W''$ and
$w^\lambda_0$ the longest element in $W_\lambda$,
then $\mathrm{GKdim}(L(w''_0w^\lambda_0\cdot \lambda))$
can be computed by a very easy formula in 
\cite[Lemma~9.15(a)]{Ja}. Namely, 
$\mathrm{GKdim}(L(w''_0w^\lambda_0\cdot \lambda))$ equals the
number of positive roots for $W$ (note: it is really 
$W$ and not $W_\lambda$) minus the number of 
positive roots for $W''$. In fact, if $W_\lambda$ is of type
$A$, then any KL-two-sided cell of $W_\lambda$ contains some
element of the form $w''_0w^\lambda_0$ and hence the
above applies. For a singular weight $\mu$, the Gelfand-Kirillov
dimension of the corresponding simple highest weight module
equals the  Gelfand-Kirillov
dimension of the simple highest weight module
for a regular correspondent of $\mu$.

\section{Structure of $\theta L$}\label{s3}

\subsection{Quick recap of birepresentation theory}\label{s3.3}

Let $\mathscr{C}$ be a finitary bicategory with
involution and adjunctions (fiab bicategory), see \cite{MMMTZ}.
Recall that a finitary birepresentation $\mathbf{M}$ of $\mathscr{C}$
is called {\em transitive} provided that it is generated,
as a birepresentation of $\mathscr{C}$, by any non-zero object.
Furthermore, $\mathbf{M}$ is called {\em simple transitive}
provided that it does not have any proper non-trivial $\mathscr{C}$-invariant
ideals.

Recall the cell theory for bicategories, see \cite{MM1,MMMTZ}. 
For two indecomposable $1$-morphisms $\theta,\theta'$ in $\mathscr{C}$,
we write $\theta\geq_L\theta'$ provided that there is $\theta''$ in 
$\mathscr{P}$ such that $\theta$ is isomorphic to a summand of 
$\theta''\circ \theta'$. Equivalence classes with respect to 
the pre-order $\geq_L$ are called {\em left cells}. Right and
two-sided cells are defined similarly. 
This is similar to the combinatorics of KL-cells which 
was recalled in Subsection~\ref{s2.65}.
Each transitive birepresentation has an {\em apex}, that is the
maximum two-sided cell whose $1$-morphisms do not annihilate
this birepresentation, see \cite[Subsection~3.2]{ChM}.
For two simple transitive birepresentations $\mathbf{M}$
and $\mathbf{N}$ of $\mathscr{C}$ we denote by 
$\mathrm{Dext}(\mathbf{M},\mathbf{N})$ the set of 
{\em discrete extensions} from $\mathbf{M}$ to $\mathbf{N}$,
see \cite[Subsection~5.2]{ChM}. The set 
$\mathrm{Dext}(\mathbf{M},\mathbf{N})$ is defined as 
the set of all non-empty subsets $\Theta$ of the set 
of  isomorphism classes of indecomposable 1-morphisms 
in $\mathscr{C}$, for which there exists a short exact sequence
\begin{displaymath}
0\to\tilde{\mathbf{N}}\to\mathbf{K}\to\tilde{\mathbf{M}}\to0 
\end{displaymath}
of birepresentations of $\mathscr{C}$ (in the sense of 
\cite[Subsection~5.2]{ChM}) such that 
\begin{itemize}
\item $\tilde{\mathbf{N}}$ is transitive with simple 
transitive quotient $\mathbf{N}$;
\item $\tilde{\mathbf{M}}$ is transitive with simple 
transitive quotient $\mathbf{M}$;
\item the set $\Theta$ consists of all
$1$-morphisms $\mathrm{F}\in\mathscr{C}$, for which there is a 
non-zero object $X\in \tilde{\mathbf{M}}$ such that 
$\mathrm{F}X$ has a non-zero summand from $\tilde{\mathbf{N}}$.
\end{itemize}
More generally, discrete extensions between transitive 
representations are defined as discrete extensions between 
the corresponding simple transitive  quotients.

Note that, in the above definition,  the birepresentation
$\mathbf{K}$ has exactly two weak Jordan-H{\"o}lder constituents.
The fact that $\mathrm{Dext}(\mathbf{M},\mathbf{N})=\varnothing$
means that, in any $\mathbf{K}$ as above, the additive closure of
all indecomposable objects which are not killed by 
projecting onto $\tilde{\mathbf{M}}$
is invariant under the action of $\mathscr{C}$. In particular,
for any $1$-morphism $\mathrm{F}\in\mathscr{C}$, 
the action of the Grothendieck class $[F]$ on the 
the split Grothendieck group of $\mathbf{K}$ is given by a block 
diagonal matrix with two blocks, one corresponding to the action on
the split Grothendieck group of $\tilde{\mathbf{N}}$  and  the other 
one corresponding to the action on the split Grothendieck group 
of $\tilde{\mathbf{M}}$.

\begin{lemma}\label{lem-8-1}
Let $\mathscr{C}$ be a fiab bicategory and $\mathbf{M}$
and $\mathbf{N}$ two transitive birepresentations of 
$\mathscr{C}$ with the same apex $\mathcal{J}$.
Then $\mathrm{Dext}(\mathbf{M},\mathbf{N})=\varnothing$.
\end{lemma}

\begin{proof}
We use the idea in the proofs of \cite[Corollary~20]{KM} and 
\cite[Corollary~14]{ChM}. Let $e$ be the idempotent in the
real algebra $A_\mathcal{J}$ from \cite[Subsection~9.3]{KM}, 
whose existence was proved in  \cite[Proposition~18]{KM}.
Let 
\begin{displaymath}
0\to  \mathbf{N}\to  \mathbf{K}\to  \mathbf{M}\to 0
\end{displaymath}
be a short exact sequence of birepresentations,
see \cite[Subsection~5.2]{ChM}. 

The algebra $A_\mathcal{J}$ acts on the split Grothendieck group
of $\mathbf{K}$ with coefficients in $\mathbb{R}$. The matrix of $e$,
written in the basis of indecomposable objects in
$\mathbf{N}$ and $\mathbf{M}$ has the form
\begin{equation}\label{eq57}
\left(\begin{array}{cc}A&B\\0&C\end{array}\right), 
\end{equation}
where $A$ and $C$ are real idempotent matrices with
positive entries and $B$ has non-negative real entries.

Recall that \cite[Formula~(2)]{Fl}
provides the following normal form for idempotent matrices 
with non-negative coefficients:
\begin{displaymath}
\left(\begin{array}{cccc}J&JX&0&0\\0&0&0&0\\YJ&YJX&0&0\\0&0&0&0\end{array}\right),
\end{displaymath}
where $J$ is a block diagonal matrix with diagonal blocks $J_1,J_2,\dots,J_k$,
with each $J_i$ being an idempotent matrix of rank one with non-negative
coefficients. In this normal form, any non-zero off-diagonal entry
for which  both diagonal correspondents are non-zero belongs
to one of the blocks $J_i$. If we assume that $B$ has a non-zero 
entry, we thus obtain that the whole matrix \eqref{eq57} must be
one block $J_i$, which  contradicts the fact that $J_i$
has rank one. Therefore $B=0$.

Alternatively, one can note that 
$\left(\begin{array}{cc}A&B\\0&C\end{array}\right)^2=\left(\begin{array}{cc}A&B\\0&C\end{array}\right)$
is equivalent to $A^2=A$, $C^2=C$ and $AB+BC=B$. Multiplying the last equation by $A$ 
on the left, gives $AAB + ABC = AB$, which implies that
$ABC=0$, since $A^2=A$. As all entries in both $A$ and $C$ are positive
and in $B$ are non-negative,
the equality $ABC=0$ is equivalent to $B=0$.
\end{proof}

\begin{corollary}\label{cor-8-2}
Let $\mathscr{C}$ be a fiab bicategory and  $\mathbf{M}$
a finitary birepresentation of $\mathscr{C}$ such that
all transitive subquotients of $\mathbf{M}$ have the same apex.
Then the objects of each transitive subquotient of $\mathbf{M}$
form a subbirepresentation.
\end{corollary}

\begin{proof}
By the weak Jordan-H{\"o}lder Theorem, see \cite[Theorem~8]{MM16}, there 
is a short exact sequence of birepresentations
\begin{displaymath}
0\to  \mathbf{K}\to  \mathbf{M}\to  \mathbf{N}\to 0,
\end{displaymath}
such that $\mathbf{N}$ is transitive and the number of transitive subquotients 
of $\mathbf{K}$ is one less than the number of transitive subquotients of $\mathbf{M}$.
By induction, the transitive subquotients of $\mathbf{K}$  are all subbirepresentations. 
We need to prove that the additive closure $\mathbf{N}'$ in $\mathbf{M}$ of all 
indecomposables whose image in $\mathbf{N}$ is non-zero is a subbirepresentation.

Assume that this is not the case. Consider the additive closure $\mathbf{M}'$ of
$\mathscr{C}\mathbf{N}'$ in $\mathbf{M}$. By our assumption,
$\mathbf{M}'\neq \mathbf{N}'$. Let $\mathbf{K}'$ be the additive closure in
$\mathbf{M}'$ of all indecomposable objects outside $\mathbf{N}'$.
Let $\mathbf{K}''$ be some transitive quotient of $\mathbf{K}'$ and let
$\mathbf{I}$ be the corresponding kernel. This gives rise to a short
exact sequence of birepresentations
\begin{displaymath}
0\to  \mathbf{K}''\to  \mathbf{M}'/(\mathbf{I})\to  \mathbf{N}'/(\mathbf{I})\to 0.
\end{displaymath}
By construction, this gives a non-trivial discrete extension between the simple
transitive quotients  of 
$\mathbf{N}'/(\mathbf{I})$ and $\mathbf{K}''$, which contradicts Lemma~\ref{lem-8-1}.
The claim follows.
\end{proof}

\subsection{The bicategory of projective functors}\label{s3.1}

Denote by $\mathscr{P}$ the locally finitary 
(in the sense  of \cite{Mac1,Mac2}) bicategory defined as
follows:
\begin{itemize}
\item the objects of $\mathscr{P}$ are  $\mathtt{i}_\chi$,
where $\chi$ is a central character of $U(\mathfrak{g})$;
\item $1$-morphisms in $\mathscr{P}(\mathtt{i}_\chi,\mathtt{i}_{\chi'})$
are all projective functors which restrict to functors from 
$\mathscr{O}_\chi$ to $\mathscr{O}_{\chi'}$;
\item $2$-morphisms in $\mathscr{P}(\mathtt{i}_\chi,\mathtt{i}_{\chi'})$
are natural transformations of functors,
\end{itemize}
where all identities and compositions are defined in the obvious way.

Since projective functors can be viewed as Harish-Chandra bimodules,
the bicategory $\mathscr{P}$ inherits from $\mathscr{H}$ a
$\Lambda/\Xi$-grading. For $\xi\in \Lambda/\Xi$, we denote by
$\mathscr{P}^\xi$ the corresponding homogeneous component.
In particular, $\mathscr{P}^\Xi$ is a subbicategory of
$\mathscr{P}$. 

Now let $L$ be a simple $\mathfrak{g}$-module with central character $\chi$.
For a central character $\chi'$, denote by $\mathbf{X}_{\chi'}^L$ the additive closure
in $\mathfrak{g}$-mod of all objects of the form $\theta L$, where
$\theta\in \mathscr{P}(\mathtt{i}_\chi,\mathtt{i}_{\chi'})$. Then
the collection of all these $\mathbf{X}_{\chi'}^L$ carries a natural action of
$\mathscr{P}$. In other words, we get  a birepresentation of 
$\mathscr{P}$, which we denote by $\mathbf{X}^L$. This birepresentation is locally finitary 
(cf. \cite{Mac1,Mac2}) in the sense
that it has the properties described by the following proposition.

\begin{proposition}\label{prop5}
Each $\mathbf{X}_{\chi'}^L$ is an idempotent split additive category 
with finite dimensional morphism spaces and finitely many 
isomorphism classes of indecomposable objects.
\end{proposition}

\begin{proof}
Let $\theta,\theta'\in  \mathscr{P}(\mathtt{i}_\chi,\mathtt{i}_{\chi'})$.
Then, by adjunction,
\begin{displaymath}
\mathrm{Hom}_{\mathfrak{g}}(\theta L, \theta' L)\cong
\mathrm{Hom}_{\mathfrak{g}}((\theta')^*\theta L, L),
\end{displaymath}
where $(\theta')^*$ is the biadjoint of $\theta'$.
As explained in Subsection~\ref{s2.3}, the right hand side
is finite dimensional. The rest now follows from the definitions.
\end{proof}

For a fixed $\chi'$, we have the bicategory 
$\mathscr{P}_{\chi'}:=\mathscr{P}(\mathtt{i}_{\chi'},\mathtt{i}_{\chi'})$.
This bicategory is finitary in the sense of \cite{MMMTZ} and 
$\mathbf{X}_{\chi'}^L$ is a finitary birepresentation of this bicategory.
For $\chi'=\chi_{{}_0}$, the central character of the trivial
$\mathfrak{g}$-module, the bicategory $\mathscr{P}_{\chi_{{}_0}}$ is biequivalent
to the bicategory of Soergel bimodules over the coinvariant algebra
of $W$, cf. \cite{MMMTZ2}.

Combinatorics of the action  of projective functors 
on category $\mathscr{O}$ is governed
by the Kazhdan-Lusztig basis of the Hecke algebra. Therefore
the cell  structure of the latter, which was recalled in 
Subsection~\ref{s2.65}, is just a special case of the cell structure of 
$\mathscr{P}_{\chi_{{}_0}}$. We also note that the action of projective functors
on category $\mathscr{O}$ is a right action.  With this in mind,
the properties recalled in Subsection~\ref{s2.65} can now be
reformulated in the setup of  projective functors as follows:
The annihilator $\mathrm{Ann}_{U(\mathfrak{g})}(L)$
corresponds to a right cell in $\mathscr{P}$, say $\mathcal{R}$. 
Let $\mathcal{J}$ be the two-sided cell containing $\mathcal{R}$. 
Then, for any indecomposable $\theta$, the inequality $\theta L\neq 0$ 
implies $\theta\leq_J\mathcal{J}$.

For any central character $\chi'$, we denote by $\mathbf{Y}_{\chi'}^L$ the
additive closure of all $\theta L$, where we take
$\theta\in \mathscr{P}(\mathtt{i}_\chi,\mathtt{i}_{\chi'})\cap \mathcal{J}$.
Then $\mathbf{Y}_{\chi'}^L$ is a full subcategory of $\mathbf{X}_{\chi'}^L$,
and the collection of all these $\mathbf{Y}_{\chi'}^L$ is closed under the 
action of $\mathscr{P}$. We denote the corresponding birepresentation of
$\mathscr{P}$ by $\mathbf{Y}^L$.

\begin{conjecture}\label{conj7}
The  birepresentation $\mathbf{Y}^L$ is simple transitive.
\end{conjecture}

\subsection{A partial pre-order}\label{s3.2}

Consider the set $\mathrm{Irr}(\mathfrak{g})$ of isomorphism
classes of simple $\mathfrak{g}$-modules. For 
$X,Y\in \mathrm{Irr}(\mathfrak{g})$, write 
$X\triangleright Y$ provided that there is a 
finite dimensional $\mathfrak{g}$-module $V$ such that
$V\otimes_{\mathbb{C}} X\tto Y$.  Note that the relation 
$X\triangleright Y$ implies the equality $\mathrm{GKdim}(X)=\mathrm{GKdim}(Y)$,
see Lemma~\ref{lem-n4}.

\begin{lemma}\label{lem7-5}
The relation $\triangleright$ is reflexive and transitive
(and hence is a partial pre-order).
\end{lemma}

\begin{proof}
To prove reflexivity, we can take $V$ to be the trivial module. 
To prove transitivity, assume that $V\otimes_{\mathbb{C}} X\tto Y$
and $V'\otimes_{\mathbb{C}} Y\tto Z$. Then, by exactness of projective functors,
\begin{displaymath}
(V'\otimes_{\mathbb{C}} V)\otimes_{\mathbb{C}} X\cong
V'\otimes_{\mathbb{C}} (V\otimes_{\mathbb{C}} X)\tto
V'\otimes_{\mathbb{C}} Y \tto Z.
\end{displaymath}
This completes the proof.
\end{proof}

By adjunction, it follows that the relation opposite to $\triangleright $ 
is given by the requirement that 
$X\hookrightarrow V^*\otimes_{\mathbb{C}} Y$.

For a central character $\chi$, let $\mathrm{Irr}(\mathfrak{g})_\chi$
denote the set of all simple $\mathfrak{g}$-modules with central character 
$\chi$. Then we have
\begin{displaymath}
\mathrm{Irr}(\mathfrak{g})=\coprod_{\chi}\mathrm{Irr}(\mathfrak{g})_\chi.
\end{displaymath}

For a fixed $L\in \mathrm{Irr}(\mathfrak{g})$ denote by $\mathcal{X}_L$
the set of all $L'\in \mathrm{Irr}(\mathfrak{g})$ for which there exists
a finite set of elements $L=L_1,L_2,\dots,L_k=L'\in \mathrm{Irr}(\mathfrak{g})$
such that, for each  $i$, we either have $L_i\triangleright L_{i+1}$
or $L_{i+1}\triangleright L_{i}$.

\begin{proposition}\label{prop7-7}
Let $L\in \mathrm{Irr}(\mathfrak{g})$ and $\chi$ be a central character.
Then $\mathcal{X}_L\cap \mathrm{Irr}(\mathfrak{g})_\chi$ is finite.
\end{proposition}

\begin{proof}
Without loss of generality, we may assume $L\in \mathrm{Irr}(\mathfrak{g})_\chi$.
Let $V$ be a finite dimensional $\mathfrak{g}$-module such that all 
indecomposable projective endofunctors of
$\mathscr{M}_\chi$ are direct summands of $V\otimes_{\mathbb{C}}{}_-$.
In order to prove our proposition, it is enough to show that 
any element of $\mathcal{X}_L\cap \mathrm{Irr}(\mathfrak{g})_\chi$
is a subquotient of $V\otimes_{\mathbb{C}}L$, since all elements
of $\mathcal{X}_L\cap \mathrm{Irr}(\mathfrak{g})_\chi$ have the
same Gelfand-Kirillov dimension as $L$ and there can only be finitely
many of them by the additivity of the Bernstein number,
see Subsection~\ref{s2.6}. In fact, since 
$V\otimes_{\mathbb{C}}{}_-$ already contains all projective endofunctors
of $\mathscr{M}_\chi$, it is enough to show that 
any element of $\mathcal{X}_L\cap \mathrm{Irr}(\mathfrak{g})_\chi$
is a subquotient of $V'\otimes_{\mathbb{C}}L$, for some
finite dimensional $\mathfrak{g}$-module $V'$.

Let $L'\in \mathcal{X}_L\cap \mathrm{Irr}(\mathfrak{g})_\chi$.
Then, by definition, there exists a finite set of elements 
$L=L_1,L_2,\dots,L_k=L'\in \mathrm{Irr}(\mathfrak{g})$ such that, 
for each  $i$, we either have $L_i\triangleright L_{i+1}$
or $L_{i+1}\triangleright L_{i}$. We prove the above claim by induction on $k$,
with the case $k=1$ being obvious.

For the induction step, we assume that $L_{k-1}$ is a 
subquotient of $V'\otimes_{\mathbb{C}}L$, for some
finite dimensional $\mathfrak{g}$-module $V'$.

Suppose $L_{k-1}\triangleright L_k=L'$, that is 
$V''\otimes_{\mathbb{C}}L_{k-1}\tto L_k$, for some 
finite dimensional $\mathfrak{g}$-module $V''$.
Then, by exactness, $L_k$ is a subquotient of 
$V''\otimes_{\mathbb{C}}(V'\otimes_{\mathbb{C}}L)$
and the latter is isomorphic to 
$(V''\otimes_{\mathbb{C}} V')\otimes_{\mathbb{C}}L$.
 
Suppose now $L'=L_{k}\triangleright L_{k-1}$, that is 
$V''\otimes_{\mathbb{C}}L_{k}\tto L_{k-1}$, for some 
finite dimensional $\mathfrak{g}$-module $V''$.
Then, by adjunction,
$L_{k}\hookrightarrow (V'')^*\otimes_{\mathbb{C}}L_{k-1}$
and again, by exactness, $L_k$ is a subquotient of 
$(V'')^*\otimes_{\mathbb{C}}(V'\otimes_{\mathbb{C}}L)$, and hence
of $((V'')^*\otimes_{\mathbb{C}} V')\otimes_{\mathbb{C}}L$.
The claim follows. 
\end{proof}

Note that $\mathcal{X}_L\cap \mathrm{Irr}(\mathfrak{g})_\chi$ is often empty.
Indeed, by \cite[Theorem~5.1]{Ko}, if $\chi'$ is the central character of $L$
and $\mathcal{X}_L\cap \mathrm{Irr}(\mathfrak{g})_\chi$ is not empty,
then there exist dominant weights $\lambda$ and $\mu$ with the following properties: 
$\chi=\chi_{{}_\lambda}$ and $\chi'=\chi_{{}_\mu}$ such that the difference $\lambda-\mu$
is an integral weight.

\begin{theorem}\label{thm7-9}
Let $L\in \mathrm{Irr}(\mathfrak{g})$ and $\chi$ be a central character.
Assume that $\mathcal{X}_L\cap \mathrm{Irr}(\mathfrak{g})_\chi$ is 
non-empty and the restriction of $\triangleright$ to it
is an equivalence relation. Then, for any central character 
$\chi'$, the restriction of  $\triangleright$ to
$\mathcal{X}_L\cap \mathrm{Irr}(\mathfrak{g})_{\chi'}$ is also an 
equivalence relation. In fact, $\triangleright$ is an equivalence relation
on $\mathcal{X}_L$.
\end{theorem}

\begin{proof}
Let  $\chi'$ be a central character such that 
$\mathcal{X}_L\cap\, \mathrm{Irr}(\mathfrak{g})_{\chi'}$
is not empty. Let $L_1,\dots,L_k$ be the list of all simples
in $\mathcal{X}_L\cap \mathrm{Irr}(\mathfrak{g})_\chi$; in particular,
they all are $\triangleright$-equivalent (and hence are also equivalent 
with respect to the relation that is opposite to $\triangleright$).
Let $L'\in \mathcal{X}_L\cap \mathrm{Irr}(\mathfrak{g})_{\chi'}$.
Then all $L_i$ (and only they) appear both in the tops and in the socles
of modules in $\mathscr{P}(\mathtt{i}_{\chi'},\mathtt{i}_{\chi})L'$.
In particular,  $L_i\triangleright L'$, for all $i$.
By adjunction, we also have $\mathrm{Hom}_\mathfrak{g}(L_i,\theta L')=
\mathrm{Hom}_\mathfrak{g}(\theta^*L_i,L')$ which implies that 
$L_i\triangleright L'$, for all $i$.
The claim follows.
\end{proof}

\begin{conjecture}\label{conj7-2}
The relation $\triangleright$ is an equivalence relation.
\end{conjecture}

\begin{remark}
{\em
It is also natural to consider the partial pre-order 
$\to$ on $\mathrm{Irr}(\mathfrak{g})$ defined as follows:
$L\to L'$ provided that $L'$ is a subquotient of 
$V\otimes_\mathbb{C}L$, for some finite-dimensional
$V$. It would be interesting to understand certain properties,
in particular, the equivalence classes of  this pre-order.
For example, for simple highest weight modules in 
category $\mathscr{O}$, the corresponding equivalence classes
are given by the KL-right cells. Also, the restriction of
$\triangleright$ to simple highest weight modules  is
an equivalence relation and the corresponding equivalence classes
are given by the KL-right cells (so they coincide with the 
equivalence classes for the pre-order $\to$).
}
\end{remark}

\subsection{Conjecture~\ref{conj7} vs Conjecture~\ref{conj7-2}}\label{s3.4}

\begin{theorem}\label{thm7-11}
Let $L$ be a simple $\mathfrak{g}$-module such that 
the restriction of $\triangleright$ to $\mathcal{X}_L$
is an equivalence relation. Then the birepresentation
$\mathbf{Y}^L$ is transitive. Moreover, we have
$\mathbf{Y}^L=\mathbf{Y}^{L'}$, for any $L'$ such that 
$L\triangleright L'$.
\end{theorem}

\begin{proof}
Let $\chi$ be the central character of $L$ and let
$\chi'$ be some central character such that 
$\mathbf{Y}^L_{\chi'}$ is not zero. Denote by
$\tilde{\mathscr{P}}$ the $1$-full and $2$-full subbicategory of
$\mathscr{P}$ on the objects $\mathtt{i}_\chi$
and $\mathtt{i}_{\chi'}$. Also denote by 
$\tilde{\mathbf{Y}}^L$ the birepresentation of
$\tilde{\mathscr{P}}$ restricted from $\mathbf{Y}^L$.
To prove the first part of the theorem, it is enough 
to show that $\tilde{\mathbf{Y}}^L$ is transitive.

Being a finitary birepresentation of $\tilde{\mathscr{P}}$,
the birepresentation
$\tilde{\mathbf{Y}}^L$ has a weak Jordan-H{\"o}lder series
with transitive subquotients.

Let us now assume that $\tilde{\mathbf{Y}}^L$ is not transitive and 
let the additive closure $\mathcal{M}$ of some
$M_1,M_2,\dots,M_k$ be a transitive subbirepresentation of 
$\tilde{\mathbf{Y}}^L$. 
Let $L'$ be a simple module which appears in the top of $M_1$.
Consider the corresponding $\tilde{\mathbf{Y}}^{L'}$ and let
the additive closure $\mathcal{N}$ of some
$N_1,N_2,\dots,N_r$ be a transitive subbirepresentation of 
$\tilde{\mathbf{Y}}^{L'}$.  By Corollary~\ref{cor-8-2},
any transitive subquotient of $\tilde{\mathbf{Y}}^L$
gives, in fact, a subbirepresentation, and similarly for 
$\tilde{\mathbf{Y}}^{L'}$. Hence,
to prove our theorem, it is
enough to show that $\mathcal{M}=\mathcal{N}$.

Indeed, as $\mathcal{N}$ is arbitrary, 
$\mathcal{M}=\mathcal{N}$ implies that 
$\tilde{\mathbf{Y}}^{L'}=\mathcal{N}$ is transitive.
Since the restriction of $\triangleright$ to $\mathcal{X}_L$
is an equivalence relation, swapping the roles of
$L$ and $L'$ we obtain that $\tilde{\mathbf{Y}}^L$
is transitive, a contradiction.
Also, from $\mathcal{M}=\mathcal{N}$ we obtain 
$\tilde{\mathbf{Y}}^L=\tilde{\mathbf{Y}}^{L'}$.

The remainder of the proof is dedicated to showing that $\mathcal{M}=\mathcal{N}$.
Applying projective functors to $M_1\tto L'$, we obtain that
every object in $\tilde{\mathbf{Y}}^{L'}$ is a quotient of
an object in $\mathcal{M}$. In particular, every object in 
$\mathcal{N}$ is a quotient of an object in $\mathcal{M}$.

Now recall that we have assumed that the restriction of 
$\triangleright$ to $\mathcal{X}_L$ is an equivalence relation.
This implies that $L$ is a quotient of some object in 
$\mathcal{N}$, say $N_1$.
Applying projective functors to $N_1\tto L$, we obtain that
every object in $\tilde{\mathbf{Y}}^{L}$ is a quotient of
an object in $\mathcal{N}$. In particular, every object in 
$\mathcal{M}$ is a quotient of an object in $\mathcal{N}$.

This implies the existence of an infinite sequence of surjections
\begin{equation}\label{eq2}
\dots \tto Y_2 \tto X_2\tto Y_1 \tto  X_1\tto N_1\tto L,
\end{equation}
where all $X_i\in \mathcal{M}$ and all $Y_j\in \mathcal{N}$.
Now, in each $Y_j$ we can pick an indecomposable summand
$N_{s_j}$ such that the restricted map from 
$N_{s_j}$ to $L$ is a surjection. Since the number of indices
for $N_j$'s is finite, we can pick an infinite 
subsequence of the form $\dots \to N_p\to N_p\to N_p\to L$.
Again, here, at each position, the map from $N_p$ to $L$ is 
a surjection, in particular, all maps between all components
of this sequence are non-zero. 

The endomorphism algebra of $N_p$ is a local finite dimensional
algebra, see Proposition~\ref{prop5}, and hence its 
Jacobson radical is nilpotent of a fixed finite nilpotency
degree. Since the above sequence is infinite and all 
compositions are non-zero, at least one morphism in this
sequence does not belong to the Jacobson radical and hence
is invertible. This means that in the original sequence 
\eqref{eq2}, we have a fragment of the form
$N_p\to X_i\to N_p$ such that the composition from the 
left to the right is invertible. Hence $N_p$ is isomorphic 
to a summand of $X_i$. In other words, $\mathcal{M}$
and $\mathcal{N}$ have a non-zero intersection and thus
must coincide since both carry a transitive birepresentations of
$\tilde{\mathscr{P}}$. This completes the proof.
\end{proof}

\begin{remark}\label{rem-8-3}
{\em
If $\mathbf{Y}^L=\mathbf{Y}^{L'}$, for any $L'\in\mathcal{X}_L$,
then the restriction of $\triangleright$ to $\mathcal{X}_L$
is an equivalence relation. Indeed, in this case, we claim that
$L\triangleright L'$ implies $L'\triangleright L$. To see this,
we first claim that $\mathbf{Y}^L$ contains a module with top $L$.

Consider the Duflo involution $\theta$ in the right cell
that corresponds to the annihilator of $L$. This is a coalgebra
$1$-morphism in $\mathscr{P}$, see \cite[Section~4.4]{MMMTZ2}, 
and hence the evaluation of the
counit $\theta\to\mathbbm{1}$, when applied to $L$, is non-zero.

Similarly, $\mathbf{Y}^{L'}$ contains a module with top $L'$.
Therefore $\mathbf{Y}^L=\mathbf{Y}^{L'}$ together with
$L\triangleright L'$ implies $L'\triangleright L$.
}
\end{remark}

\section{Proof of the two conjectures in type $A$}\label{s4}

In this section, we show that the statements of 
both Conjecture~\ref{conj7} and Conjecture~\ref{conj7-2}
are true in type $A$. So, we assume that $\mathfrak{g}$
and hence also $W$ are of type $A$.

\subsection{Reduction to nice blocks}\label{s4.05}

Let $\lambda$ be a dominant weight and $\chi_{{}_\lambda}$ the corresponding
central character. We will call both $\lambda$ and $\chi_{{}_{\lambda}}$
{\em nice} provided that there is  $\mu\in \lambda+\Xi$ such that
$W_\mu=W'_\mu$. For example, in the set $\mathbb{Z}$ of all integral
weights for $\mathfrak{sl}_2$, we have $\Xi=2\mathbb{Z}$ and
all odd weights are nice (since $-1$
is the only integral singular weight), while all even weights are not nice.
Note that $W_\lambda=\{e\}$  implies that $\lambda$ is nice as we  can take
$\mu=\lambda$.

\begin{lemma}\label{lem-nice}
For any  dominant weight $\lambda$, there is a nice dominant
weight $\tilde{\lambda}\in\lambda+\Lambda$ such that 
\begin{enumerate}[$($a$)$]
\item \label{lem-nice.1} $W_\lambda=W_{\tilde{\lambda}}$,
\item \label{lem-nice.2} $W'_\lambda=W'_{\tilde{\lambda}}$,
\item \label{lem-nice.3} $\tilde{\lambda}-\lambda$ is integral  and dominant
with respect to $W_\lambda$.
\end{enumerate}
\end{lemma}

\begin{proof}
We first  prove the claim under the assumption that $\lambda$
is integral. The weight $-\rho$ is the only integral fully singular weight,
so we need to look for $\tilde{\lambda}$ inside $-\rho+\Xi$.
Let $D$ be the absolute value of the determinant of 
the Cartan matrix of $\mathfrak{g}$. Set
$\tilde{\lambda}=D(\lambda+\rho)-\rho$.
Since the $D$-multiples of the fundamental weights belong to
$\Xi$, it follows that $\tilde{\lambda}\in-\rho+\Xi$.

Then $\tilde{\lambda}-\lambda=(D-1)(\lambda+\rho)$ which is dominant.
Both $\lambda$ and $\tilde{\lambda}$ are integral and hence, for both
of them, the integral Weyl group is just the whole Weyl group.
Finally, the stabilizers of $\tilde{\lambda}$ and $\lambda$
in $W$ with respect to the dot action coincide because, after the
shift by $\rho$, the dot-action becomes the usual action and this
commutes with multiplication by scalars. This proves the claim
for integral weights.

Take now any $\lambda$ and assume $W_\lambda\neq \{e\}$,
for otherwise the claim is clear. Let $\mathbf{R}_\lambda$ be the root
subsystem of $\mathbf{R}$ corresponding to $W_\lambda$.
Let $\mathfrak{g}(\lambda)$  be the corresponding Lie subalgebra
of $\mathfrak{g}$.  Let $\Lambda(\lambda)$ be the set of all integral weights for
$\mathfrak{g}(\lambda)$ and $\Xi(\lambda)$ the set of all 
integral linear combinations of roots for $\mathfrak{g}(\lambda)$.
Choose some representatives $\mu_1=0,\mu_2,\dots,\mu_k$
of the cosets in $\Lambda(\lambda)/\Xi(\lambda)$.
Also, denote by $\mathfrak{h}_\lambda$ the intersection of
$\mathfrak{h}$  with $\mathfrak{g}(\lambda)$. Define
$\mathfrak{h}_\lambda^\perp$ as the set of all 
$h\in \mathfrak{h}$ such that $\alpha(h)=0$, for any root
$\alpha$ of $\mathfrak{g}(\lambda)$. Then
$\mathfrak{h}=\mathfrak{h}_\lambda\oplus \mathfrak{h}_\lambda^\perp$.
The inclusion $\mathfrak{h}_\lambda\hookrightarrow \mathfrak{h}$
induces the restriction map $\mathrm{Res}_\lambda:\mathfrak{h}^*\to \mathfrak{h}_\lambda^*$.

The restriction of the natural 
$\mathfrak{g}$-module (i.e.  the module $\mathbb{C}^n$
for $\mathfrak{sl}_n$) to any simple summand $\mathfrak{a}$
of $\mathfrak{g}(\lambda)$
gives the direct sum of the natural module for $\mathfrak{a}$
with a summand on which $\mathfrak{a}$ acts trivially. 
Recall that the natural module generates the category of all 
finite dimensional modules as an idempotent split monoidal category.
This implies that
\begin{equation}\label{eq-res-5}
\mathrm{Res}_\lambda(\lambda+\rho+\Lambda)=
\mathrm{Res}_\lambda(\lambda)+ \Lambda(\lambda)+\rho_\lambda,
\end{equation}
where $\rho_\lambda$ is the half of the sum of all  positive roots for
$\mathfrak{g}(\lambda)$. We note that both $\rho$ and $\rho_\lambda$
are integral weights, so they can be removed from \eqref{eq-res-5}.
In particular,  we can pick some representatives
$\nu_1=\lambda,\nu_2,\dots,\nu_k$ in $\lambda+\Lambda$
such that $\mathrm{Res}_\lambda(\nu_i+\rho-\lambda)=\mu_i+\rho_\lambda$.

We can now apply the already proved assertion of the lemma in the integral case
to $\mathrm{Res}_\lambda(\lambda)$ to obtain the corresponding nice dominant integral weight
$\widetilde{\mathrm{Res}_\lambda(\lambda)}$ for $\mathfrak{g}(\lambda)$
which satisfies \eqref{lem-nice.1}, \eqref{lem-nice.2} and \eqref{lem-nice.3}
(with respect to $\mathrm{Res}_\lambda(\lambda+\rho)-\rho_\lambda$ for $\mathfrak{g}(\lambda)$).
By \eqref{eq-res-5}, there is $\tilde{\lambda}\in  \lambda+\Lambda$
such that $\mathrm{Res}_\lambda(\tilde{\lambda}+\rho)-\rho_\lambda=\widetilde{\mathrm{Res}_\lambda(\lambda)}$.
The fact that  $\tilde{\lambda}$ is nice dominant and  satisfies \eqref{lem-nice.1}, \eqref{lem-nice.2}
and \eqref{lem-nice.3} follows  from the fact that 
$\widetilde{\mathrm{Res}_\lambda(\lambda)}$ is nice dominant and 
satisfies \eqref{lem-nice.1}, \eqref{lem-nice.2} and \eqref{lem-nice.3}.
\end{proof}

\begin{example}\label{ex-new-lem20}
{\em
Consider $\mathfrak{g}=\mathfrak{sl}_3$ 
with $\mathbf{R}=\{\pm\alpha,\pm\beta,\pm(\alpha+\beta)\}$.
With respect to the standard basis, we then have
\begin{displaymath}
\alpha=\left(\begin{array}{c}2\\-1\end{array}\right),\quad
\beta=\left(\begin{array}{c}-1\\2\end{array}\right),\quad
\alpha+\beta=\rho=\left(\begin{array}{c}1\\1\end{array}\right).
\end{displaymath}
Consider the weight
$\lambda=\left(\begin{array}{c}x-1/2\\-x-1/2\end{array}\right)$,  for some 
irrational $x$. Then we  have $\mathbf{R}_\lambda=\{\pm(\alpha+\beta)\}$
and  $\rho_\lambda=(1)$. This yields 
$\mathrm{Res}_\lambda(\lambda+\rho)-\rho_\lambda=0$, 
which is not a nice $\mathfrak{sl}_2$-weight.
Therefore $\lambda$ is not nice. 

To get a nice weight,
we  add to $\lambda$ the integral weight 
$\lambda=\left(\begin{array}{c}1\\0\end{array}\right)$
resulting in the weight 
$\tilde{\lambda}=\left(\begin{array}{c}x+1/2\\-x-1/2\end{array}\right)$.
We have
$\mathrm{Res}_\lambda(\tilde{\lambda}+\rho)-\rho_\lambda=1$, 
which is a nice $\mathfrak{sl}_2$-weight. Clearly,
the conditions \eqref{lem-nice.1}, \eqref{lem-nice.2} and \eqref{lem-nice.3}
of Lemma~\ref{lem-nice} are satisfied for this $\tilde{\lambda}$.

One can also easily find a singular weight  in $\tilde{\lambda}+\Xi$.
For example, the weight
$\tilde{\lambda}-(\alpha+\beta)=\left(\begin{array}{c}x-1/2\\-x-3/2\end{array}\right)$
is singular.
One checks  that $\mathrm{Res}_\lambda({}_-+\rho)-\rho_\lambda$
maps this weight to $-1$,  namely, to the unique singular  $\mathfrak{sl}_2$-weight.
\hfill\qed
}
\end{example}

One consequence of the above lemma is that, 
if $\lambda$ is nice, then $\lambda+\Xi$ contains dominant weights
of arbitrary singularity in $W_{\lambda}$.

If $\lambda$ and $\tilde{\lambda}$ are as above, then
$\theta_{\lambda,\tilde{\lambda}}$ and
$\theta_{\tilde{\lambda},\lambda}$ are mutually inverse
equivalences of categories, both at the level of category
$\mathscr{O}$ and at the level of category $\mathscr{M}$.
In particular, these equivalences send simple objects to
simple objects. Consequently, for any simple 
$\mathfrak{g}$-module $L$ in $\mathcal{M}_\chi$,
the categories $\mathrm{add}(\mathscr{P}L)$
and $\mathrm{add}(\mathscr{P}\theta_{\lambda,\tilde{\lambda}}(L))$
coincide (in the sense that they have the same
objects and morphisms).

Since both Conjectures~\ref{conj7} and \ref{conj7-2}
are formulated in terms of $\mathrm{add}(\mathscr{P}L)$,
it follows that it is enough to prove them for simple
modules with nice central characters.

\subsection{Reduction to singular blocks}\label{s4.1}
%

Let $L$ be a simple $\mathfrak{g}$-module, $\chi$ be the central character
of $L$, which we assume to be nice, 
and $\lambda\in\mathfrak{h}^*$ be some weight such that
$\mathrm{Ann}_{U(\mathfrak{g})}(L)=\mathrm{Ann}_{U(\mathfrak{g})}(L(\lambda))$.
We have the bicategory $\mathscr{P}_\chi$ of all projective
endofunctors of $\mathscr{M}_\chi$.

Consider the integral Weyl group $W_\lambda$ of $\lambda$.
Then $W_\lambda$ is a product of symmetric groups. Let us start by recalling 
special features of Kazhdan-Lusztig combinatorics in type $A$.
Thanks to \cite[Theorem~1.4]{KL}, in type $A$, left and right cells of
$\mathscr{P}^\Xi$ can be described using the Robinson-Schensted correspondence
which associates to a permutation $w\in S_n$ a pair 
of standard Young tableaux of the same shape (which is a partition of $n$). 
The latter shape determines the two-sided cell. One special type $A$ feature
is that each two-sided  cell $\mathcal{J}$ of 
$\mathscr{P}^\Xi$ in type $A$ contains the longest element
$\theta_{w_0^{W'}}$ in some parabolic subgroup $W'$ of $W_\lambda$. Another special
feature is that the intersection of any left and any right cell in $W_\lambda$
is a singleton. This means that $\mathcal{J}$ contains the identity functor
$\theta_e^{\chi'}$ on some $\mathscr{M}_{\chi'}$ 
(a singular block for which $W'$ is the 
dot-stabilizer of the dominant weight for that block) 
and this functor is the only projective endofunctor of
$\mathscr{M}_{\chi'}$ belonging to the intersection of the cell $\mathcal{J}$
with the homogeneous component $\mathscr{P}_\chi^\Xi$.
From Lemma~\ref{lem-n853} below it follows that inclusion gives rise to
a bijection between the left (right, two-sided) cells of $\mathscr{P}^\Xi$
and the corresponding cells of $\mathscr{P}$.

The elements in $\mathcal{J}$ that do not annihilate $L$ 
are exactly the elements of the left cell which is  adjoint
to the right cell which corresponds to the annihilator of $L$,
see \cite[Lemma~12]{MM1}. Each of
these left cells contains an element with target $\mathscr{M}_{\chi'}$.
We choose one of those, call it $\theta$.  Then $\theta L\neq 0$ and we 
can let $L'\in \mathscr{M}_{\chi'}$ be any simple quotient of $\theta L$.
Since $\chi$ is assumed to be nice, $\theta$ is homogeneous of degree $\Xi$.

We note that, by  construction, the identity projective 
functor $\theta_e^{\chi'}$ on $\mathscr{M}_{\chi'}$ is the only 
indecomposable projective  endofunctor of $\mathscr{M}_{\chi'}$ 
that belongs to $\mathscr{P}_\chi^\Xi$ and does not annihilate $L'$.

Consider the annihilator $\mathbf{I}:=\mathrm{Ann}_{U(\mathfrak{g})}(L')$
of $L'$ in $U(\mathfrak{g})$. Let $\mathscr{M}^{\mathbf{I}}_{\chi'}$
denote the full subcategory of $\mathscr{M}_{\chi'}$ which consists of
all objects on which $\mathbf{I}$ acts locally nilpotently. Then 
$\theta_e^{\chi'}$ is still the identity endofunctor 
of $\mathscr{M}^{\mathbf{I}}_{\chi'}$ and it does not annihilate
$L'$.

We want to answer the following question: What are the other projective  
endofunctors of $\mathscr{M}^{\mathbf{I}}_{\chi'}$ 
that do not annihilate $L'$? 

Choose some $\lambda'$ such that $\mathrm{Ann}_{U(\mathfrak{g})}(L')=
\mathrm{Ann}_{U(\mathfrak{g})}(L( \lambda'))$, which is possible 
due to Duflo's Theorem, see \cite{Du}. Consider $W\cdot\lambda'$ and its
intersections with all $\Xi$-cosets in $\lambda'+\Lambda$. Those cosets
in $\Lambda/\Xi$ for which the intersection is non-trivial form a subgroup of
the cyclic group $\Lambda/\Xi$. Let $\mu_1$, $\mu_2$,\dots, $\mu_k$
be the dominant weights in all the corresponding non-empty intersections
(of $W\cdot\lambda'$ with the $\Xi$-cosets in $\lambda'+\Lambda$).
Without loss of generality, we may assume $\lambda'\in W_{\mu_1}\cdot\mu_1$.
We have, $\theta_e^{\chi'}=\theta_{\mu_1,\mu_1}$.

The integral Weyl groups $W_{\mu_i}$ are all conjugate and so are
the stabilizers of the corresponding dominant weights in these
$W_{\mu_i}$, see also \cite[Remark~3.5]{Hu}. In particular, 
from Soergel's combinatorial description it follows that 
all the corresponding indecomposable blocks 
$\mathscr{O}_{\mu_i}$ 
(see Subsection~\ref{s2.2}) of category $\mathscr{O}$ are equivalent, see 
also \cite[Lemma~A.3]{Mat} for an alternative argument.
In particular, all $L(\mu_i)$ have the same Gelfand-Kirillov dimension,
see Subsection~\ref{s2.65}.

Since our $\mu_i$ might be singular, the category $\mathscr{O}_{\mu_i}$
might contain some other simple highest weight modules $L(\nu)$ with 
the same Gelfand-Kirillov dimension as $L(\mu_i)$. In this case 
we will write $\nu\sim\mu_i$.

\begin{lemma}\label{lem-n853}
Each indecomposable projective functor that does not annihilate $L'$ 
is of the form $\theta_{\mu_1,\nu}$, where $\nu\sim \mu_i$, for some $i$,
and $\nu$ is anti-dominant with respect to the dot-stabilizer of 
$\mu_1$. 
Moreover, each such $\theta_{\mu_1,\nu}$ is a
self-equivalence of $\mathscr{M}^{\mathbf{I}}_{\chi'}$
(but not necessarily of $\mathscr{M}_{\chi'}$).
\end{lemma}

\begin{proof}
We have $\theta_{\mu_1,\nu} L'\neq 0$ if and only if 
$\theta_{\mu_1,\nu} L(\lambda')\neq 0$, by our choice of $\lambda'$.
Since projective functors cannot increase the Gelfand-Kirillov dimension
and all $L(\mu_i)$ have the same Gelfand-Kirillov dimension,
$\theta_{\mu_1,\nu} L'\neq 0$ implies $\nu\sim \mu_i$, for some $i$,
by our definition of $\sim$. From the classification of projective
functors, we may also assume that $\nu$ is anti-dominant with 
respect to the dot-stabilizer of $\mu_1$. It remains to argue that 
any such $\theta_{\mu_1,\nu}$ is an equivalence.

Let $\theta_{\mu_1,\nu}^*$ denote the biadjoint of $\theta_{\mu_1,\nu}$.
We have $\theta_{\mu_1,\nu} P(\mu_1)=\theta_{\mu_1,\nu} \Delta(\mu_1)=P(\nu)$, 
by the classification of projective functors. In particular,
\begin{displaymath}
\dim\mathrm{Hom}_\mathfrak{g}(\theta_{\mu_1,\nu} \Delta(\mu_1),L(\nu))=1. 
\end{displaymath}
By adjunction, we thus have
\begin{displaymath}
\dim\mathrm{Hom}_\mathfrak{g}(\Delta(\mu_1),\theta_{\mu_1,\nu}^*L(\nu))=1. 
\end{displaymath}
Note that $L(\mu_1)$ is the simple top of $\Delta(\mu_1)$ and it appears
in $\Delta(\mu_1)$ with multiplicity one.

Assume that the image of a unique (up to scalar) non-zero map from 
$\Delta(\mu_1)$ to $\theta_{\mu_1,\nu}^*L(\nu)$ is not isomorphic to 
$L(\mu_1)$. Then the socle of this image contains some $L(\nu')$,
where $\nu'\sim\mu_1$ and $\nu'\neq\mu_1$. As $\theta_{\mu_1,\nu}^*L(\nu)$
is self-dual
(since projective functors commute with the duality 
on $\mathscr{O}$), $L(\nu')$ also appears in its top. This means that the
composition $\theta_{\mu_1,\nu}^*\circ \theta_{\mu_1,\nu}$ applied
to $\Delta(\mu_1)$ has $P(\nu')$ as a summand. However, this is not 
possible as $\nu'\sim\mu_1$ and $\theta_{\mu_1,\mu_1}$ is the only
projective endofunctor of $\mathscr{O}_{\mu_1}$ that does not kill
$L(\mu_1)$.

From the previous paragraph, we have that the image of a unique 
(up to scalar) non-zero map from  $\Delta(\mu_1)$ to $\theta_{\mu_1,\nu}^*L(\nu)$ 
is isomorphic to  $L(\mu_1)$, in particular, $L(\mu_1)$ appears in
the socle of $\theta_{\mu_1,\nu}^*L(\nu)$. As
$\theta_{\mu_1,\nu}^*L(\nu)$ is self-dual, 
$L(\mu_1)$ appears in
the top of $\theta_{\mu_1,\nu}^*L(\nu)$ as well. Since  $\Delta(\mu_1)$ is 
projective and the map from it to $\theta_{\mu_1,\nu}^*L(\nu)$ is unique,
the multiplicity of $L(\mu_1)$ in $\theta_{\mu_1,\nu}^*L(\nu)$ is one. 
Consequently, $L(\mu_1)$ is a direct summand of $\theta_{\mu_1,\nu}^*L(\nu)$.
As, by the previous paragraph, no other $L(\nu')$ with $\nu'\sim\mu_1$
are allowed to appear in the top or socle of $\theta_{\mu_1,\nu}^*L(\nu)$,
we have $\theta_{\mu_1,\nu}^*L(\nu)=L(\mu_1)$. By adjunction,
$\theta_{\mu_1,\nu}L(\mu_1)=L(\nu)$.

This implies that
\begin{displaymath}
\theta_{\mu_1,\nu}^*\theta_{\mu_1,\nu}=
\theta_{\mu_1,\nu}\theta_{\mu_1,\nu}^*=
\theta_{\mu_1,\mu_1}
\end{displaymath}
as endofunctors of $\mathscr{M}^{\mathbf{I}}_{\chi'}$ and completes the proof.
\end{proof}

\begin{example}\label{ex-new-sl_2}
{\em
For $\mathfrak{g}=\mathfrak{sl}_2$, consider $\lambda=(-1/2)$. 
Then $W\cdot\lambda=\{-1/2,-3/2\}$.
Both latter weights are dominant with respect to  
their integral Weyl group, which is trivial. However, the difference 
between these two  weights is an integral
weight. Therefore we have two indecomposable projective 
functors that do not annihilate $L(-1/2)$, namely, the
identity functor $\theta_{-1/2,-1/2}$ and the equivalence 
$\theta_{-1/2,-3/2}$ between $\mathscr{O}_{-1/2}$
and $\mathscr{O}_{-3/2}$.\hfill\qed
} 
\end{example}

\subsection{Proof of Conjecture~\ref{conj7-2} in type $A$}\label{s4.2}

From the construction in the previous subsection, it follows that 
$\mathcal{X}_L\cap \mathrm{Irr}(\mathfrak{g})_{\chi'}$
consists of modules of the form $\theta_{\mu,\mu'}(L')$,
where $\theta_{\mu,\mu'}$ are equivalences. This set is,
clearly, one equivalence class with respect to $\triangleright$. 
Therefore the claim of 
Conjecture~\ref{conj7-2} follows from Theorem~\ref{thm7-9}.

\subsection{Proof of Conjecture~\ref{conj7} in type $A$}\label{s4.3}

We now establish a crucial property of the module $L'$
constructed in Subsection~\ref{s4.1}, namely, its weak Kostant positivity
(see Subsection~\ref{s2.4}):

\begin{lemma}\label{lem08}
The module $L'$ is weakly Kostant positive. 
\end{lemma}

\begin{proof}
Since we are discussing only the weak Kostant positivity
of $L'$ in this lemma, in the proof below we restrict our attention to 
indecomposable projective functors that are homogeneous of degree $\Xi$.
Recall that $\theta_e^{\chi'}$ is the only indecomposable projective 
endofunctor of $\mathscr{M}_{\chi'}$ that is homogeneous of degree $\Xi$
and does not annihilate $L'$, see Section~\ref{s2new}.

We start with the claim that $\lambda'$ may be assumed to be 
dominant (i.e. we may assume
$\lambda'=\mu_1$).
In other words, we claim that
$\mathrm{Ann}_{U(\mathfrak{g})}(L')=\mathrm{Ann}_{U(\mathfrak{g})}(L(\mu_1))$.
To prove this, we need to show that $L(\lambda')$ is annihilated by
all indecomposable projective endofunctors of $\mathscr{O}_{\chi'}$ 
that are not isomorphic to the identity functor.

Let $\theta$ be an indecomposable projective endofunctor of 
$\mathscr{M}_{\chi'}$, homogeneous of degree $\Xi$, 
which is not isomorphic to the identity functor.
Then $\theta P(\mu_1)\cong P(\nu)$, for some $\nu\neq \mu_1$.
Since $\theta$ is strictly bigger than $\mathcal{J}$ in the 
two-sided order, the Gelfand-Kirillov dimension of $L(\nu)$
is strictly greater than that of $L(\mu_1)$, see
\cite[Subsection~10.11]{Ja}. From \cite[Subsection~10.9]{Ja}
it then follows that $\theta L(\mu_1)=0$. Therefore 
the annihilator of $L(\mu_1)$ corresponds to a right
cell inside $\mathcal{J}$ and since the right cell of 
$\theta_e^{\chi'}$ is the only right cell that contains a
representative from
projective endofunctors of $\mathscr{O}_{\chi'}$ in 
$\mathcal{J}$, we obtain that the annihilator of $L(\mu_1)$
corresponds to the right cell of $\theta_e^{\chi'}$.
This is the same right cell which describes the annihilator of
$L'$, and we obtain our claim.

Now, assuming $\lambda'$ is dominant, 
$L(\lambda')$ is the quotient of a
projective Verma 
module in $\mathscr{O}_{\chi'}$. Hence it is Kostant positive,
see \cite[Subsection~6.9]{Ja}.
For a simple finite dimensional $V$, 
the multiplicity of $V$ in 
$\mathcal{L}(L(\lambda'),L(\lambda'))$ equals the dimension of
$\mathrm{Hom}_{\mathfrak{g}}(V\otimes_{\mathbb{C}}L(\lambda'),L(\lambda'))$.
We can write $V\otimes_{\mathbb{C}}{}_-$ as a direct sum of indecomposable projective
functors. The summands which go from $\mathscr{O}_{\chi'}$ to 
$\mathscr{O}_{\chi'}$ are either $\theta_e^{\chi'}$ or kill $L(\lambda')$.
Therefore the above multiplicity equals the multiplicity of 
$\theta_e^{\chi'}$ as a summand of $V\otimes_{\mathbb{C}} {}_-$.

The same computation works for $L'$, under the assumption that 
$0$ is a weight of $V$ (which is equivalent to saying that all
indecomposable projective functors that appear as summands
of $V\otimes_{\mathbb{C}}{}_-$ are homogeneous of degree $\Xi$), 
which is equivalent to saying that 
all indecomposable projective functors 
that are summands of ${}_-\otimes_{\mathbb{C}} V$ are homogeneous
of degree $\Xi$. Therefore, combining
\begin{displaymath}
U(\mathfrak{g})/\big(\mathrm{Ann}_{U(\mathfrak{g})}(L')\big)
\cong 
U(\mathfrak{g})/\big(\mathrm{Ann}_{U(\mathfrak{g})}(L(\lambda'))\big)
\end{displaymath}
with
\begin{displaymath}
U(\mathfrak{g})/\big(\mathrm{Ann}_{U(\mathfrak{g})}(L(\lambda'))\big)
\cong \mathcal{L}(L(\lambda'),L(\lambda')),
\end{displaymath}
then with 
\begin{displaymath}
U(\mathfrak{g})/\big(\mathrm{Ann}_{U(\mathfrak{g})}(L')\big)
\hookrightarrow
\mathcal{L}(L',L')^\Xi,
\end{displaymath}
and, finally, with
\begin{displaymath}
[\mathcal{L}(L(\lambda'),L(\lambda')):V]
= [\mathcal{L}(L',L')^\Xi:V],
\end{displaymath}
we obtain
\begin{displaymath}
U(\mathfrak{g})/\big(\mathrm{Ann}_{U(\mathfrak{g})}(L')\big)
\cong \mathcal{L}(L',L')^\Xi.
\end{displaymath}
This completes the proof.
\end{proof}

Denote by ${}^\Xi \mathbf{X}^{L'}$ and ${}^\Xi \mathbf{Y}^{L'}$ the
$\mathscr{P}^\Xi$-analogues of $\mathbf{X}^{L'}$ and $\mathbf{Y}^{L'}$,
respectively.

Since $L'$ is weakly Kostant positive  and the identity projective 
functor $\theta_e^{\chi'}$ on $\mathscr{M}_{\chi'}$ is the only 
indecomposable projective  endofunctor of $\mathscr{M}_{\chi'}$ 
that belongs to $\mathscr{P}_\chi^\Xi$ and does not annihilate $L'$, we can apply 
the adaptation \cite[Theorem~5]{KhM} of
\cite[Theorem~5.1]{MiSo}  and conclude that ${}^\Xi \mathbf{X}^{L'}$
is equivalent to a certain category of Harish-Chandra bimodules.
This equivalence is, in fact, a homomorphism of 
birepresentations of $\mathscr{P}^\Xi$.

We can also apply \cite[Theorem~5]{KhM} to
$L(\lambda)$ and conclude that ${}^\Xi \mathbf{X}^{L(\lambda)}$
is equivalent to the same  category of Harish-Chandra bimodules,
as birepresentations of $\mathscr{P}^\Xi$. In other words,
${}^\Xi \mathbf{X}^{L'}$ and ${}^\Xi \mathbf{X}^{L(\lambda)}$ are
equivalent as birepresentations of $\mathscr{P}^\Xi$.
This means that such an equivalence induces an equivalence
between ${}^\Xi \mathbf{Y}^{L'}$ and ${}^\Xi \mathbf{Y}^{L(\lambda)}$.

Since we already established 
in Subsection~\ref{s4.2} that the statement of 
Conjecture~\ref{conj7-2} holds for $\mathcal{X}_L$,
Theorem~\ref{thm7-11} implies that $\mathbf{Y}^{L}=
\mathbf{Y}^{L'}$. This means that ${}^\Xi \mathbf{Y}^{L}$
and ${}^\Xi \mathbf{Y}^{L(\lambda)}$ are equivalent as birepresentations
of $\mathscr{P}^\Xi$. It is easy to see that both $\mathbf{Y}^{L(\lambda)}$ 
and ${}^\Xi \mathbf{Y}^{L(\lambda)}$ are simple
transitive by combining \cite[Theorem~22]{MM1}, 
\cite[Proposition~22]{MM2} and \cite[Theorem~18]{MM16}.
Hence ${}^\Xi \mathbf{Y}^{L}$ is simple transitive.

The elements of $\mathscr{P}$ inside $\mathcal{J}$ 
are obtained from the elements of $\mathscr{P}^\Xi$
inside $\mathcal{J}$ by composing with the auto-equivalences
given by Lemma~\ref{lem-n853}. Hence $\mathbf{Y}^{L}$
is obtained from ${}^\Xi \mathbf{Y}^{L}$ by applying some 
equivalences. Since $\mathbf{Y}^{L}$ is transitive by 
construction, the fact that it is simple transitive 
follows from the simple transitivity of ${}^\Xi \mathbf{Y}^{L}$.
This completes the proof.

We remark that some of the arguments in this subsection,
in particular, the idea of reduction to a singular block,
are similar in spirit to the arguments given in the 
proof of \cite[Theorem~67]{MS08}.

\subsection{Some corollaries}\label{s4.4}

The equivalence between ${}^\Xi \mathbf{Y}^{L}$
and ${}^\Xi \mathbf{Y}^{L(\lambda)}$ established in the previous subsection
has the following consequence:

\begin{corollary}\label{cor-2.4.4.1}
Let $\theta$ be an indecomposable projective functor from  
$\mathcal{J}$ and $M$ a subquotient of
$V\otimes_{\mathbb{C}}L$ such that $\theta M\neq 0$.
Then $\mathrm{GKdim}(M)=\mathrm{GKdim}(L)$. 
\end{corollary}

\begin{proof}
The point is that in $\mathbf{Y}^{L(\lambda)}$
any simple subquotient $M$ of any $V\otimes_{\mathbb{C}}L(\lambda)$
satisfying $\theta M\neq 0$ appears in the top  of
some $V'\otimes_{\mathbb{C}}L(\lambda)$. 
We know this because we understand the action of 
projective functors on category $\mathscr{O}$
quite well. Applying to $L(\lambda)$ projective functors
from $\mathcal{J}$ produces a cell birepresentation
of the bicategory of projective functors.
In the abelianization of this birepresentation,
the action of projective functors from $\mathcal{J}$
is given by tensoring with projective bimodules which are 
explicitly described in \cite[Proposition~4.15]{MMMTZ}.
Applying such a projective module does exactly what is 
claimed above: applied to a summand of
$V\otimes_{\mathbb{C}}L(\lambda)$ in which $M$
appears as a subquotient, it produces a module
with  a direct summand in which $M$ appears in the top.

Due to the 
equivalence between ${}^\Xi \mathbf{Y}^{L}$
and ${}^\Xi \mathbf{Y}^{L(\lambda)}$ and the fact that 
$\mathbf{Y}^{L}$ is obtained from ${}^\Xi \mathbf{Y}^{L}$
(resp. $\mathbf{Y}^{L(\lambda)}$ from ${}^\Xi \mathbf{Y}^{L(\lambda)}$)
using some equivalences of categories given by projective functors,
it follows that $M$ has  a subquotient which appears in the top  of
some $V''\otimes_{\mathbb{C}}L$. This implies
$\mathrm{GKdim}(M)=\mathrm{GKdim}(L)$. 
\end{proof}

Let $L$ be a simple $\mathfrak{g}$-module and $V$ a finite dimensional
$\mathfrak{g}$-module. A subquotient $M$ of $V\otimes_\mathbb{C}L$ will
be called {\em strange} provided that the following conditions are
satisfied:
\begin{itemize}
\item for any submodule $N\subset M$, exactly one of the modules
$N$ or $M/N$ has GK-dimension $\mathrm{GKdim}(L)$,
\item $M$ does not have any simple subquotient of 
GK-dimension $\mathrm{GKdim}(L)$.
\end{itemize}
This definition is inspired by the properties of the
regular $\mathbb{C}[x]$-module. This module and all its 
non-zero submodules have GK-dimension $1$,
while any quotient of this module by a non-zero 
submodule has GK-dimension $0$. The module itself does not
have any simple submodules.

\begin{corollary}\label{cor-2.4.4.2}
Let $\theta$ be an indecomposable projective functor from  
$\mathcal{J}$ and $M$ a strange subquotient of
$V\otimes_{\mathbb{C}}L$. Then $\theta M=0$. 
\end{corollary}

\begin{proof}
As explained in the proof of Corollary~\ref{cor-2.4.4.1},
the assumption $\theta M\neq 0$ implies that $M$ has  a subquotient which 
appears in the top  of some $V''\otimes_{\mathbb{C}}L$.
In other words, $M$  has a simple subquotient $M'$ such that 
$\mathrm{GKdim}(M')=\mathrm{GKdim}(L)$. This contradicts
our assumption that $M$ is strange.
\end{proof}

\subsection{The two conjectures in other type $A$ situations}\label{s4.5}

Let now $\mathfrak{g}$ be of any type. Let $L$ be a simple
$\mathfrak{g}$-module and $\lambda\in\mathfrak{h}^*$ be such that
the annihilators  of $L$ and $L(\lambda)$ coincide.
Combining the above results with Soergel's combinatorial
description, it follows that 
both Conjecture~\ref{conj7} and Conjecture~\ref{conj7-2}
are true for $L$ under the assumption that $W_\lambda$
is of type $A$. More precisely, we have:

\begin{corollary}\label{cor-s4.4-1}
Assume that $W_\lambda$
is of type $A$. Then the  birepresentation $\mathbf{Y}^L$ 
is simple transitive and the restriction of the 
relation $\triangleright$ to $\mathcal{X}_L$ is the full relation
(i.e.  any two elements are related).
\end{corollary}


\section{Socles}\label{s2}

\subsection{The main result}\label{s2.1}

Let us start with repeating the formulation of the main result.

\begin{theorem}\label{thm1}
Let $\mathfrak{g}$ be a semi-simple finite dimensional Lie algebra
over $\mathbb{C}$. Let $L$ be a holonomic simple $\mathfrak{g}$-module and
let $V$ be a finite dimensional $\mathfrak{g}$-module. Then the 
$\mathfrak{g}$-module $V\otimes_{\mathbb{C}} L$ has an
essential semi-simple submodule of finite length.
\end{theorem}

\subsection{Reduction to a finite piece}\label{s2.7}

Let $L$ be a simple $\mathfrak{g}$-module and $\lambda\in\mathfrak{h}^*$
be such that $\mathrm{Ann}_{U(\mathfrak{g})}(L)=
\mathrm{Ann}_{U(\mathfrak{g})}(L(\lambda))$. Consider the set
$\lambda+\Lambda$ and  the set 
$\mathcal{K}(\lambda)$ of all central characters of the form
$\chi_{{}_\mu}$, where $\mu\in \lambda+\Lambda$. 

As we have seen in Subsection~\ref{s2.5}, the 
combinatorial datum which  controls equivalences
between blocks of $\mathscr{O}$ is given  by
triples of the form $W'_\mu\subset W_\mu\subset W$,
for $\mu\in \mathfrak{h}^*$
(see \cite{Kevin} for an explicit classification).
Therefore, it is natural to consider the 
finite set of all triples of the form 
$\tilde{\tilde{W}}\subset \tilde{W}\subset W$, where
$\tilde{W}$ is the subgroup of $W$ generated by some reflections
(and hence is the Weyl group of the root subsystem of $\mathbf{R}$
generated by the roots corresponding to those reflections) and
$\tilde{\tilde{W}}$ is a parabolic subgroup of $\tilde{W}$ (with respect to the choice of 
positive roots inherited from  $\mathbf{R}_+$). 

Given two dominant weights $\mu$ and $\nu$ in $\lambda+\Lambda$,
we have their respective integral Weyl groups $W_\mu$ and $W_\nu$
and their respective dot-stabilizers $W'_\mu$ and $W'_\nu$.
If $W_\mu=W_\nu$ and $W'_\mu=W'_\nu$, then the projective
functors $\theta_{\mu,\nu}$ and $\theta_{\nu,\mu}$ are mutually
inverse equivalences of categories.

Now we can fix a finite set $\mu_1$, $\mu_2$,\dots, $\mu_k$
of dominant weights in $\lambda+\Lambda$ such that
$\mu_1\in W_\lambda\cdot \lambda$ and, for any 
other dominant weight $\nu$ in $\lambda+\Lambda$, there is
some $\mu_i$ such that we have both the equality $W_{\mu_i}=W_\nu$
of the corresponding integral Weyl groups and the equality
$W'_{\mu_i}=W'_\nu$ of the corresponding dot-stabilizers. 
Let $\mathscr{N}$ be the direct sum of
all the corresponding $\mathscr{M}_{\chi_{{}_{\mu_i}}}$. 

\begin{proposition}\label{prop-reduction}
In order to prove Theorem~\ref{thm1},
it is enough to prove that, for any indecomposable projective
endofunctor $\theta$ of $\mathscr{N}$, the module
$\theta L$ has an essential semi-simple submodule of finite length.
\end{proposition}

\begin{proof}
Due to our construction of $\mathscr{N}$, any non-zero projective 
functor $\theta'$ from $\mathscr{M}_{\chi_{{}_\lambda}}$ to some
$\mathscr{M}_{\chi_{{}_{\nu}}}$ factors through some
$\mathscr{M}_{\chi_{{}_{\mu_i}}}$ via equivalences of categories
given by projective functors. Equivalences of categories,
clearly, preserve the module theoretic property of having 
an essential semi-simple submodule of finite length.
\end{proof}

\subsection{
Reduction to the maximal two-sided cell}\label{s2.9}

Let  $L\in \mathscr{N}$ be a simple module and $\theta$ a projective
endofunctor of $\mathscr{N}$. Let $\mathcal{J}$ be the two-sided KL-cell
that contains the left KL-cell corresponding to the annihilator of
$L$ in $U(\mathfrak{g})$. Let $\theta_\mathcal{J}$ be a multiplicity-free
direct sum of all projective endofunctors of $\mathscr{N}$ that belong
to $\mathcal{J}$. Let $\tilde{\theta}$ be the Duflo element in the
left KL-cell corresponding to the annihilator of
$L$ in $U(\mathfrak{g})$. Then we have a natural transformation 
from the identity to $\tilde{\theta}$, whose evaluation at $L$ is non-zero,
see \cite[Subsection~4.5]{MM1}. Consequently, $L$ appears as a 
submodule of $\tilde{\theta} L$, and hence as a submodule of
$\theta_\mathcal{J} L$, since $\tilde{\theta}$ is a summand of
$\theta_\mathcal{J}$. 

Applying $\theta$ to the inclusion $L\hookrightarrow \theta_\mathcal{J} L$, 
we get an inclusion of $\theta L$ into $\theta \theta_\mathcal{J} L$.
The composition $\theta \theta_\mathcal{J}$ belongs to the additive closure
of $\theta_\mathcal{J}$, modulo projective functors from strictly higher
two-sided cells. The latter projective functors annihilate $L$ because
of our assumption on the annihilator of $L$. This means that $\theta L$
is a submodule of $\theta' L$, for some $\theta'$ in the additive closure of
$\theta_\mathcal{J}$. Therefore, if we can prove Theorem~\ref{thm1}
for $\theta= \theta_\mathcal{J}$, it follows that Theorem~\ref{thm1}
is true for all $\theta$.

\subsection{Proof of Theorem~\ref{thm1}}\label{s2.8}

Unfortunately, $\mathscr{M}$ does not have arbitrary products,
which is a technical obstacle for our coming arguments 
that we need to deal with.

For $k\in\mathbb{Z}_{>0}$ and a central character $\chi$, 
denote by $\mathscr{M}^k_\chi$ the full subcategory of 
$\mathscr{M}_\chi$ that consists of all modules annihilated
by the $k$-th power of the kernel of $\chi$. Let 
$\mathscr{M}^k$ be the product of all $\mathscr{M}^k_\chi$.
Then, by \cite[Theorem~5.1]{Ko}, for a projective functor $\theta$
and $k\in\mathbb{Z}_{>0}$, there is $m\in\mathbb{Z}_{>0}$
such that $\theta$ maps $\mathscr{M}^k$ to $\mathscr{M}^m$.
Note that $\mathscr{M}^k$ has arbitrary limits, for all $k$.
%
%
%
%

Now let $L\in \mathscr{N}$ be a simple module and $\theta$ a projective
endofunctor of $\mathscr{N}$ which belongs to 
$\mathcal{J}$. As we only have finitely
many indecomposable projective endofunctors of $\mathscr{N}$,
we can fix $V$ such that all 
indecomposable projective endofunctors of $\mathscr{N}$  are direct summands of the 
projective functor  $V\otimes_{\mathbb{C}}{}_-$. In particular, 
by the additivity of the Bernstein number with respect to short exact 
sequences, see Subsection~\ref{s2.6}, for any filtration
\begin{displaymath}
0=M_0\subset M_1\subset \dots \subset M_k=  \theta L,
\end{displaymath}
the number of $i$ such that $\mathrm{GKdim}(M_i/M_{i-1})=\mathrm{GKdim}(L)$
cannot be greater than $\dim(V)\cdot \mathrm{BN}(L)$.

Let $N$ be a maximal semi-simple submodule of $\theta L$.
We know, see Subsection~\ref{s2.6}, that it has finite length. 
Assume that it is not essential
and let $K$ be a non-zero submodule of $\theta L$ such that
$K\cap N=0$. Then $K$ has no simple submodule.
From the previous paragraph, we may further assume that 
any quotient of $K$ by a non-zero submodule has 
Gelfand-Kirillov dimension strictly smaller than $\mathrm{GKdim}(L)$.
Indeed, if $K$ has a non-zero submodule $K'$ such that
$\mathrm{GKdim}(K/K')=\mathrm{GKdim}(L)$, we can simply replace $K$
by $K'$. After at most $\dim(V)\cdot \mathrm{BN}(L)$ replacements, we obtain a $K$
with the desired property. In particular, $K$ is a strange 
submodule of $\theta L$.

First of all, we note that, by adjunction,
\begin{displaymath}
0\neq \mathrm{Hom}_{\mathfrak{g}}(K,\theta L)\cong 
\mathrm{Hom}_{\mathfrak{g}}(\theta^* K,L),
\end{displaymath}
in particular, $\theta^* K\neq 0$.

On the other hand, we want to show that $\theta^* K=0$ and in this way get a contradiction.
For example, in type $A$, $\theta^* K=0$ follows from Corollary~\ref{cor-2.4.4.2}.
To prove $\theta^* K=0$ in general (but, under the additional,
compared to Corollary~\ref{cor-2.4.4.2}, assumption that $L$ is holonomic),
consider the filtered diagram $\mathcal{P}$ 
of quotients of $K$ by non-zero
submodules with respect to natural projections. The kernel of the natural map from
$K$ to the limit $\displaystyle \lim_{\leftarrow}\mathcal{P}$ equals the intersection of all
non-zero submodules of $K$. That is zero, as $K$ does not have simple submodules, 
in particular, it does not have a simple socle.

As $\theta^*$ has a biadjoint, we have $\displaystyle\theta^*\lim_{\leftarrow}\mathcal{P}
\cong\lim_{\leftarrow}\theta^*\mathcal{P}$.
At the same time, 
the Gelfand-Kirillov dimension of any $X\in \mathcal{P}$
is strictly smaller than $\mathrm{GKdim}(L)$, by our assumption on $K$. 
Since $\theta\in\mathcal{J}$, we have $\theta^*\in\mathcal{J}$ and hence
$\theta^* X=0$ by our assumption that $L$ is holonomic.
This means that $\displaystyle\lim_{\leftarrow}\theta^*\mathcal{P}=0$, 
which implies that $\theta^* K=0$, a contradiction.
This proves that such  $K$ cannot exist and 
completes the proof of Theorem~\ref{thm1}.

\section{Beyond holonomic modules outside type $A$}\label{s2.17s}

\subsection{Results}\label{s2.17}
 
As already mentioned in the introduction, in type $A$,
the  assertion of Theorem~\ref{thm1} is true for all simple
modules $L$, not necessarily holonomic ones, see
\cite[Theorem~23]{CCM21}. The main reson why this works is the
combinatorial property of type $A$ that each two-sided KL-cell
contains the longest element $w_0^\mathfrak{p}$ of
the Weyl group of some
parabolic subalgebra $\mathfrak{p}$.
We can generalize \cite[Theorem~23]{CCM21} as follows.

\begin{theorem}\label{thm1-n}
Let $\mathfrak{g}$ be a semi-simple classical finite dimensional Lie algebra
over $\mathbb{C}$. Let $L$ be a simple $\mathfrak{g}$-module 
such that the two-sided KL-cell $\mathcal{J}$ that contains the left KL-cell
corresponding to the annihilator of $L$ in $U(\mathfrak{g})$
contains some $w_0^\mathfrak{p}$.
Let $V$ be a finite dimensional $\mathfrak{g}$-module. Then the 
$\mathfrak{g}$-module $V\otimes_{\mathbb{C}} L$ has an
essential semi-simple submodule of finite length.
\end{theorem}

In the same setup, we also prove  both Conjectures~\ref{conj7} 
and \ref{conj7-2}.

\begin{theorem}\label{thm1-n2}
Let $\mathfrak{g}$ be a semi-simple classical finite dimensional Lie algebra
over $\mathbb{C}$. Let $L$ be a simple $\mathfrak{g}$-module 
such that the two-sided KL-cell $\mathcal{J}$ that contains the left KL-cell
corresponding to the annihilator of $L$ in $U(\mathfrak{g})$
contains some $w_0^\mathfrak{p}$.
For such $L$, the assertions of both Conjectures~\ref{conj7} 
and \ref{conj7-2} are true.
\end{theorem}

\subsection{Proof of Theorem~\ref{thm1-n}}\label{s2.18}

We follow the idea of the proof of \cite[Theorem~23]{CCM21},
which is also utilized, in a slightly disguised way, in Section~\ref{s4}. 
Here is a sketch of this idea:
\begin{itemize}
\item Due to our assumption on $\mathcal{J}$, we can translate $L$
to a singular block whose singularity corresponds to our longest element.
\item The indecomposable projective endofunctors of that singular 
block that do not kill $L$ form a group (modulo projective functors
that kill $L$), in particular, they are invertible. So, for all
such projective functors, the claim of Theorem~\ref{thm1-n} is straightforward.
\item The assertion of Theorem~\ref{thm1-n} is equivalent to saying
that $V\otimes_{\mathbb{C}} L$
has no strange submodules. If we assume that this is wrong, then we can translate
a strange submodule of $V\otimes_{\mathbb{C}} L$
to the singular block from above, which
leads to a contradiction with the previous item.
\end{itemize}

The first item on the above list goes mutatis mutandis as in Subsection~\ref{s4.1}.
We note that, outside type $A$, two-sided KL-cells do not have to contain 
any longest element for some parabolic subgroup. However, we explicitly assume
this for our $\mathcal{J}$, which allows us to use the approach of 
Subsection~\ref{s4.1}. This approach leads to the following output:
starting from $L$ with some central character $\chi$ corresponding
to a dominant weight $\lambda$, we find a singular weight
$\lambda'$ with the corresponding central character $\chi'$
such that the singularity $W'$ of $\lambda'$ in 
$W_{\lambda'}$ is a parabolic subgroup and is isomorphic to the Weyl group 
of our longest element in the formulation. We also find a simple module $L'$ with the 
same annihilator as $L(\lambda')$ and such that 
the additive closure of all $\theta L$, where $\theta\in\mathcal{J}$,
coincides with the the additive closure of all $\theta L'$, 
where $\theta\in\mathcal{J}$. Therefore we can forget about $L$
and concentrate on $L'$.

For the second item on the above list, let us assume that $\lambda'$
is a singular weight with singularity $W'$ (which is a parabolic subgroup
of $W_{\lambda'}$). Let $\chi'$ be the central character of $L(\lambda')$. 
Consider the bicategory $\mathscr{P}(\mathtt{i}_{\chi'},\mathtt{i}_{\chi'})$
and the biideal $\mathscr{J}_{\chi'}$ in it generated by all 
indecomposable objects that are not two-sided equivalent to the 
identity.

\begin{lemma}\label{lem-new557}
Any indecomposable object of 
$\mathscr{P}(\mathtt{i}_{\chi'},\mathtt{i}_{\chi'})/\mathscr{J}_{\chi'}$
is invertible.
\end{lemma}

\begin{proof}
Arguments similar to the ones used in the proof of Lemma~\ref{lem-n853}
reduce the necessary statement to the similar statement for 
$\mathscr{P}^\Xi(\mathtt{i}_{\chi'},\mathtt{i}_{\chi'})/
(\mathscr{P}^\Xi(\mathtt{i}_{\chi'},\mathtt{i}_{\chi'})\cap\mathscr{J}_{\chi'})$.

We can translate singular projective functors out of the wall all the way
to the corresponding regular blocks. We can also translate back. 
Translating out and then back gives $|W'|$ copies of what
we started with, with one copy in degree zero and all other copies 
shifted in positive degrees, see \cite[Proposition~4.1]{CM}.
Since the endomorphism algebra of the multiplicity-free direct sum
of all  indecomposable  objects of the bicategory of projective 
functors with tops concentrated in degree zero
is positively graded, see \cite{Ba}, it follows
that $\mathscr{P}^\Xi(\mathtt{i}_{\chi'},\mathtt{i}_{\chi'})/
(\mathscr{P}^\Xi(\mathtt{i}_{\chi'},\mathtt{i}_{\chi'})\cap\mathscr{J}_{\chi'})$
is biequivalent to the asymptotic category associated with the $H$-cell 
of $\mathcal{J}^\Xi$ which contains our longest element (for  the definition of
this asymptotic category, see e.g. \cite[Subsection~3.2]{MMMTZ2}). 

As explained in \cite[Section~8]{MMMTZ2},
since we assume $\mathfrak{g}$ to be classical, all the asymptotic 
bicategories that appear are biequivalent to the bicategory of 
finite dimensional vector spaces graded by a finite group.
In the latter, all indecomposable objects are invertible.
The claim of the lemma follows.
\end{proof}

Finally, to justify the last item on the above list, let $M$ be a strange
submodule of $V\otimes_{\mathbb{C}} L$. Then, for any projective functor 
$\theta$, the module $\theta M$ cannot have simple submodules of the same 
GK-dimension as $L$. Indeed, if $\tilde{L}$ were such a submodule,
then, by adjunction
\begin{displaymath}
0\neq \mathrm{Hom}_\mathfrak{g}(\tilde{L},\theta M)\cong
\mathrm{Hom}_\mathfrak{g}(\theta^*\tilde{L},M)
\end{displaymath}
and we would get a contradiction, since all quotients of $\theta^*\tilde{L}$
have the same GK-dimension as $L$ while all quotients of $M$
by non-zero submodules have strictly smaller GK-dimension.

Suppose that $M'$ is a non-zero submodule of  $\theta M$.
By the additivity of the Bernstein number, $M'$ can  only have finitely many 
simple subquotients of the same GK-dimension as $L$, say $L_1$, $L_2$,\dots,
$L_k$ (counted with the respective multiplicities). 
Let $I_i$ be the indecomposable injective hull of $L_i$. The embeddings
$L_i\subset I_i$ give rise to a map from $M'$ to 
$I_1\oplus I_2\oplus\dots \oplus I_k$. This map cannot be injective since
none of the $L_i$ is a submodule of $M'$ by the previous paragraph. 
The kernel of this map is thus a strange submodule of $M'$.

Consequently, any non-zero submodule of 
$\theta M$ has a strange submodule. This means that we can translate
our $M$ to our singularity $\chi'$ and obtain that 
$\theta' L'$ must have a strange submodule for some
$\theta'\in \mathscr{P}^\Xi(\mathtt{i}_{\chi'},\mathtt{i}_{\chi'})$.
At the same time, by Lemma~\ref{lem-new557}, all indecomposable
summands of $\theta'$ are invertible, which implies that 
$\theta' L'$ is semi-simple of finite length, a contradiction.
This completes the proof of Theorem~\ref{thm1-n}.

\subsection{Proof of Theorem~\ref{thm1-n2}}\label{s2.19}
 
In the setup of Theorem~\ref{thm1-n2}, the assertion of 
Conjecture~\ref{conj7-2} follows from Theorem~\ref{thm7-9}.
Indeed, if we look at the proof  of Theorem~\ref{thm1-n},
the application of projective functors to $L'$ produces
semi-simple  modules of finite length, if we restrict to the central  character $\chi'$.
The simple  constituents of these semi-simple modules, obviously,
form an equivalence class with respect to $\triangleright$.
Hence the assumptions of Theorem~\ref{thm7-9} are satisfied, so
this theorem applies.

To prove Conjecture~\ref{conj7} in the setup of Theorem~\ref{thm1-n2},
we note that Theorem~\ref{thm7-11} already guarantees that 
$\mathbf{Y}^L$ is transitive. Since the underlying category 
$\mathbf{Y}^L(\mathtt{i}_{\chi'})$ is semi-simple, 
$\mathbf{Y}^L(\mathtt{i}_{\chi'})$ is simple transitive
as a birepresentation of $\mathscr{P}(\mathtt{i}_{\chi'},\mathtt{i}_{\chi'})$. 
Since we also have $\mathbf{Y}^L=\mathbf{Y}^{L'}$ by Theorem~\ref{thm7-11}, 
any socle constituent of any object in $\mathbf{Y}^L$ can be translated,
using adjunction, back to $\mathbf{Y}^L(\mathtt{i}_{\chi'})$
in a non-zero way. If $\mathbf{Y}^L$ were not simple
transitive, the kernel of the projection from $\mathbf{Y}^L$ onto its
unique simple transitive quotient would kill some socle constituent
of some object in $\mathbf{Y}^L$. Translating to $\mathbf{Y}^L(\mathtt{i}_{\chi'})$
we would be forced to kill a non-zero object of this category,
contradicting its simple transitivity. This implies that 
already $\mathbf{Y}^L$ is simple transitive.

We note that the result proved in the previous paragraph can also 
be obtained using the results of \cite[Subsection~4.8]{MMMTZ}.

\section{Strange subquotients, Serre quotients and rough structure}\label{s9}

\subsection{Strange subquotients}\label{s9.1}

Let $L$ be a simple $\mathfrak{g}$-module and $V$ a finite dimensional
$\mathfrak{g}$-module. The module 
$V\otimes_\mathbb{C}L$, clearly, does not have any 
strange quotients. Furthermore, Theorem~\ref{thm1} essentially says 
that $V\otimes_\mathbb{C}L$ does  not have any strange submodules
provided that $L$ is holonomic.
A priori, we cannot rule out existence of strange subquotients.
However, inspired by Theorem~\ref{thm1}, we propose the following conjecture:

\begin{conjecture}\label{conj-strange}
Strange subquotients of $V\otimes_\mathbb{C}L$ do not exist. 
\end{conjecture}

Let $M$ be a strange subquotient of $V\otimes_\mathbb{C}L$.
Let $N$ denote the sum of all submodules of $M$ whose 
 GK-dimension  is strictly smaller than 
$\mathrm{GKdim}(L)$. Since $V\otimes_\mathbb{C}L$ and hence also
$M$ are noetherian, $N$ is finitely generated. Since each of the
finitely many generators of $N$ belongs to a submodule of $M$ whose 
GK-dimension is strictly smaller than 
$\mathrm{GKdim}(L)$, it follows that $\mathrm{GKdim}(N)< \mathrm{GKdim}(L)$.
Therefore the subquotient $M/N$ is also strange and has the
property that any submodule of $M/N$ has GK-dimension $\mathrm{GKdim}(L)$.
We will  say that $M/N$ is a strange subquotient in {\em normal form}.
Note that  strange subquotients in normal form do not have simple
submodules at all.

\begin{proposition}\label{prop-str2}
Let $L$ be holonomic, $M$  a strange subquotient of 
$V\otimes_\mathbb{C}L$ and $\theta$ an indecomposable projective functor
from the $\mathcal{J}$-cell corresponding to $\mathrm{Ann}_{U(\mathfrak{g})}(L)$.
Then $\theta M=0$.
\end{proposition}

\begin{proof}
This is proved using the same argument as at the end of Subsection~\ref{s2.8}.
\end{proof}

\subsection{Serre quotients}\label{s9.2}

Let $L$ be a simple $\mathfrak{g}$-module. In general, the module
$V\otimes_{\mathbb{C}} L$ need not have finite length in $\mathfrak{g}$-mod. 
In this subsection we introduce a natural subquotient of 
$\mathfrak{g}$-mod where $V\otimes_{\mathbb{C}} L$ always has
finite length and a well-defined notion of composition
multiplicities.

Let $\mathscr{A}=\mathscr{A}(L)$ denote the full subcategory of 
$\mathfrak{g}$-mod whose objects are all finitely 
generated $\mathfrak{g}$-modules isomorphic to subquotients of 
modules of the form $V\otimes_{\mathbb{C}} L$, where $V$ is
a finite dimensional $\mathfrak{g}$-module. This is an abelian 
subcategory of $\mathfrak{g}$-mod with the abelian structure
(for example, $\mathbb{C}$-linearity, kernels and cokernels) 
being inherited from $\mathfrak{g}$-mod. Thanks to exactness of
projective functors, the category $\mathscr{A}$ comes equipped 
with the natural action of projective functors.

Let $\mathscr{B}=\mathscr{B}(L)$ denote the full subcategory of 
$\mathscr{A}$ consisting of all modules of Gelfand-Kirillov
dimension strictly smaller than $\mathrm{GKdim}(L)$. From 
Subsection~\ref{s2.6} it follows that $\mathscr{B}$ is
a Serre subcategory of $\mathscr{A}$ as well as that $\mathscr{B}$ is
stable under the action of projective functors. Therefore
$\mathscr{A}/\mathscr{B}$ is an abelian category that has 
a natural action of projective functors. 

Let $\mathscr{C}=\mathscr{C}(L)$ denote the 
full  subcategory of $\mathscr{A}$ consisting of 
all objects $M$ of $\mathscr{A}$ that have the property that 
$\theta M=0$, for any indecomposable projective functor $\theta$
from the two-sided cell $\mathcal{J}$ corresponding to $\mathrm{Ann}_{U(\mathfrak{g})}(L)$.
Since projective functors are exact, $\mathscr{C}$ 
is a Serre subcategory of $\mathscr{A}$.

\begin{lemma}\label{lem-nn7}
The category $\mathscr{C}$ is stable under the action of 
projective functors.
\end{lemma}

\begin{proof}
Let $\theta'$ be a projective functor and $\theta$ be a projective functor 
from $\mathcal{J}$. Then any indecomposable summand in both $\theta\theta'$ 
and  $\theta'\theta$ is either in $\mathcal{J}$ or annihilates $L$.
This implies the claim of the lemma.
\end{proof}

Due to Lemma~\ref{lem-nn7}, the category  $\mathscr{A}/\mathscr{C}$ 
is an abelian category that has a natural action of projective functors.

\begin{conjecture}\label{conjn-9}
$\mathscr{B}=\mathscr{C}$.
\end{conjecture}

Note that, if $L$ is holonomic, we have $\mathscr{B}\subset \mathscr{C}$.

\subsection{Rough structure  of modules in $\mathscr{A}$}\label{s9.3}

\begin{theorem}\label{prop-JH}
Assume that $L$  is holonomic.

\begin{enumerate}[$($a$)$]
\item \label{prop-JH.1} The category $\mathscr{A}/\mathscr{C}$ 
is an abelian  length category.
\item \label{prop-JH.2} Simple  objects in 
$\mathscr{A}/\mathscr{C}$ are in bijection with 
isomorphism classes of simple subquotients 
with GK-dimension $\mathrm{GKdim}(L)$ in modules 
of the form  $V\otimes_{\mathbb{C}} L$, where $V$ is
a finite dimensional $\mathfrak{g}$-module.
\item \label{prop-JH.3} Every object in
$\mathscr{A}/\mathscr{C}$ has well-defined 
composition multiplicities. 
\end{enumerate}
\end{theorem}

As suggested in \cite{MS08}, for $X\in \mathscr{A}$,
the part of the structure of  $X$ which can be seen
in $\mathscr{A}/\mathscr{C}$ (including the 
multiplicities in $X$ of simple $\mathfrak{g}$-modules
with GK-dimension $\mathrm{GKdim}(L)$) is called the 
{\em rough structure} of $X$.

\begin{proof}[Proof of Theorem~\ref{prop-JH}.]
A subquotient $X$ of some  $V\otimes_{\mathbb{C}} L$
will be called {\em primitive } provided that, for any 
submodule $Y\subset X$, at most one of the modules
$Y$ or $X/Y$ has GK-dimension $\mathrm{GKdim}(L)$.

Given a primitive subquotient $X$ of some  
$V\otimes_{\mathbb{C}} L$, the definitions of
$\mathscr{A}$, $\mathscr{C}$  and the Serre quotient 
give us three options:
\begin{itemize}
\item The GK-dimension of $X$ is strictly smaller than
$\mathrm{GKdim}(L)$. In this case $X=0$ in 
$\mathscr{A}/\mathscr{C}$.
\item The module $X$  is strange. In this case $X=0$ in 
$\mathscr{A}/\mathscr{C}$.
\item The module $X$ has a unique simple subquotient
$X'$ of  GK-dimension $\mathrm{GKdim}(L)$.
In this case $X=X'$ in $\mathscr{A}/\mathscr{C}$.
\end{itemize}
This implies Claim~\eqref{prop-JH.2}.

The category $\mathscr{A}/\mathscr{C}$ is abelian by 
construction. That every object in $\mathscr{A}/\mathscr{C}$
has finite length follows from the additivity of the
Bernstein number, since it is a positive integer and the 
Bernstein number of $V\otimes_{\mathbb{C}} L$ is finite.
This implies Claim~\eqref{prop-JH.1}.

Let $X$ and $Y$ be in $\mathscr{A}$ such that 
$Y$  is a simple  $\mathfrak{g}$-module of 
GK-dimension $\mathrm{GKdim}(L)$. Let $I_Y$
be the injective envelope of $Y$ in $\mathfrak{g}$-Mod.
Let 
\begin{displaymath}
0=X_0\subset X_1\subset X_2\subset\dots \subset X_k=X 
\end{displaymath}
be a filtration of $X$ such that all subquotients are
primitive (since the number of possible  subquotient
of GK-dimension $\mathrm{GKdim}(L)$ is bounded,
such a filtration exists). Then $X_i/X_{i-1}$ has $Y$ as a subquotient
if and only if there is homomorphism from 
$X_i/X_{i-1}$ to $I_Y$. As $X_i/X_{i-1}$ is assumed to
be primitive, the dimension of 
$\mathrm{Hom}_{\mathfrak{g}}(X_i/X_{i-1},I_Y)$ equals one
(since the endomorphism algebra of $Y$ has dimension one
by Dixmier-Schur's lemma).
Therefore the composition multiplicity of $Y$  in $X$
is finite and equals the dimension of 
$\mathrm{Hom}_{\mathfrak{g}}(X,I_Y)$.
\end{proof}


\noindent
M.~M.: Center for Mathematical Analysis, Geometry, and Dynamical Systems, Departamento de Matem{\'a}tica, 
Instituto Superior T{\'e}cnico, 1049-001 Lisboa, PORTUGAL \& Departamento de Matem{\'a}tica, FCT, 
Universidade do Algarve, Campus de Gambelas, 8005-139 Faro, PORTUGAL, email: {\tt mmackaay\symbol{64}ualg.pt}

\noindent
Vo.~Ma.: Department of Mathematics, Uppsala University, Box. 480,
SE-75106, Uppsala, SWEDEN, email: {\tt mazor\symbol{64}math.uu.se}

\noindent
Va.~Mi.: School of Mathematics, University of East Anglia, Norwich, 
NR4 7TJ, UK, email: {\tt V.Miemietz\symbol{64}uea.ac.uk}

\end{document}